\documentclass[11pt,a4paper,oneside,reqno]{amsart}
\usepackage[utf8]{inputenc}

\usepackage{amsmath,amscd}
\usepackage{amssymb}
\usepackage{amsthm}
\usepackage{comment}
\usepackage{xcolor}

\usepackage{mathrsfs}
\usepackage{hyperref}

\usepackage{graphicx}
\usepackage{psfrag}
\newcommand{\Mi}{M^{\textrm{int}}}

\newcommand{\p}{\partial}

\usepackage{epstopdf}

\newtheorem{Theorem}{Theorem}[section]

\newtheorem{Lemma}[Theorem]{Lemma}

\newtheorem{Definition}[Theorem]{Definition}

\newtheorem{Proposition}[Theorem]{Proposition}

\newcommand{\abs}[1]{\lvert #1 \rvert} 
\newcommand{\norm}[1]{\lVert #1 \rVert}

\newcommand{\R}{\mathbb{R}}                    
\newcommand{\Z}{\mathbb{Z}}                    
\newcommand{\N}{\mathbb{N}}                    

\newcommand{\eps}{\varepsilon}
\newcommand{\tg}{\tilde{g}}
\newcommand{\ol}[1]{\overline{#1}}
\renewcommand{\a}{\alpha}

\newcommand{\mL}{\mathcal{L}}

\title{Conformal harmonic coordinates}

\author{Matti Lassas}
\address{Department of Mathematics and Statistics, University of Helsinki}
\email{matti.lassas@helsinki.fi}

\author[T. Liimatainen]{Tony Liimatainen}
\address{Department of Mathematics and Statistics, University of Jyv\"askyl\"a, \newline \phantom{h\,\,} Department of Mathematics and Statistics, University of Helsinki}
\email{tony.liimatainen@helsinki.fi}
\begin{document}

\begin{abstract}
We study conformal harmonic coordinates on Riemannian manifolds. These are coordinates constructed as quotients of solutions to the conformal Laplace equation. We show their existence under general conditions. We find that conformal harmonic coordinates are a close conformal analogue of harmonic coordinates.  
We prove up to boundary regularity results for conformal mappings. We show that Weyl, Cotton, Bach, and Fefferman-Graham obstruction tensors become elliptic operators in conformal harmonic coordinates if one normalizes the determinant of the metric. We give a corresponding elliptic regularity result, which includes an analytic case. We prove a unique continuation result for local conformal flatness for Bach and obstruction flat manifolds. We discuss and prove existence of conformal harmonic coordinates on Lorentzian manifolds. We prove unique continuation results for conformal mappings both on Riemannian and Lorentzian manifolds. 

\end{abstract}

\maketitle

\section{Introduction}
In this article we study a coordinate system called \emph{conformal harmonic coordinates} on Riemannian manifolds. 
Conformal harmonic coordinates were invented in a recent joint paper~\cite{LLS} by the authors. These coordinates are conformally invariant and they have similar properties to harmonic coordinates. 
One of the main motivations of this paper is to better understand these coordinates. 
This is done by proving new results in conformal geometry and by 
giving new proofs of some earlier results. The results concern conformal mappings and conformal curvature tensors such as Weyl and  Bach tensors. 
We prove the existence and regularity of conformal harmonic coordinates and characterize them. We also initiate the study of conformal harmonic coordinates on Lorentzian manifolds.

Let $(M,g)$ be a Riemannian manifold. The standard harmonic coordinates on the manifold $M$ are defined as a coordinate system $(x^1,\ldots,x^n)$ with the property that the coordinate functions $x^k$ satisfy
\[
 \Delta_gx^k=0.
\]
Here $\Delta_g$ is the Laplace-Beltrami operator $\Delta_gu=-|g|^{-1/2}\p_a\left(|g|^{1/2}g^{ab}\p_bu\right)$. We use Einstein summation convention where repeated indices are summed over. 
Harmonic coordinates have several uses in Riemannian geometry. 
In harmonic coordinates Ricci and Riemann tensors can be regarded as elliptic operators~\cite{DK}. The DeTurck trick to prove local wellposedness of the Ricci flow by fixing the diffeomorphism invariance of the flow is also related to harmonic coordinates~\cite{ChowKnopf, GrahamLee}. Harmonic coordinates have been used to study Cheeger-Gromov convergence of manifolds~\cite{JK82, Jo84, Pe87, GW88, AKKLT04}, unique continuation of Ricci curvature~\cite{GuSaBar, AH08}, geometric inverse problems~\cite{GuSaBar, LLS18} and regularity of isometries and local flatness of low regularity metrics~\cite{Ta06}.  
The notion of harmonic coordinates goes back to Einstein~\cite{Einstein} and to studies of Lorentzian geometry in general relativity. In Einstein's general relativity harmonic coordinates have for example been used to study gravitational waves and wellposedness of the Einstein equations~\cite{Wa84, Ch09}. A comprehensive list of all the applications harmonic coordinates would be quite long.

The conformal Laplacian on a Riemannian manifold $(M,g)$ of dimension $n \geq 3$ is the operator
\[
 L_g = \Delta_g + \frac{n-2}{4(n-1)}R(g).
\]
Here $\Delta_g$ is as before and $R(g)$ is the scalar curvature of the Riemannian metric $g$. The scalar curvature acts on functions by multiplication. The conformal Laplacian is conformally invariant in the sense that if $c$ is a positive function, then 
\begin{equation}
 L_{cg} u = c^{-\frac{n+2}{4}} L_g (c^{\frac{n-2}{4}} u).
\end{equation}
We refer to~\cite{LeeParker} for more properties of the conformal Laplacian.

We define conformal harmonic coordinates as quotients of solutions to the conformal Laplace equation: We say that coordinates $(Z^1,\ldots,Z^n)$ are conformal harmonic coordinates if the coordinate functions $Z^k$ are of the form
\begin{equation}\label{quotient_form}
 Z^k=\frac{f^k}{f},
\end{equation}
where $f^k$ and $f>0$ are any functions that satisfy the conformal Laplace equation
\[
 L_gf^1=\cdots=L_gf^n=L_gf=0.
\]
Conformal harmonic coordinates are conformally invariant. If $c$ is a positive function then coordinates $(Z^1,\ldots,Z^n)$ of the form~\eqref{quotient_form} satisfy the definition of conformal harmonic coordinates for both $g$ and $cg$. 
See Proposition~\ref{Z_coords_conformally_invariant}. 
Conformal harmonic coordinates were invented in~\cite{LLS} (called $Z$-coordinates in the paper) to serve as a tool to solve a conformal Calder\'on inverse conductivity problem. The beneficial property of conformal harmonic coordinates used in the paper is that these coordinates are real analytic for any manifold that is locally conformal to a real analytic manifold. 

We prove existence and regularity properties of conformal harmonic coordinates on general Riemannian manifolds in Theorem~\ref{Z-coord}, and in Theorem~\ref{Z-coord_boundary} for manifolds with boundary. These theorems state that on a neighborhood of any point on a $C^r$ smooth, $r>2$ and $r\notin \mathbb{Z}$, Riemannian manifold exists $C^{r+1}$ smooth conformal harmonic coordinates. 
If the manifold is real analytic $C^\omega$, then conformal harmonic coordinates are real analytic. The smoothness (we also use the term regularity) of conformal harmonic coordinates is conformally invariant by Proposition~\ref{Z_coords_conformally_invariant}.

We characterize conformal harmonic coordinates in Propositions~\ref{equivalence}$-$~\ref{prop:isothermal}. We notice that conformal harmonic coordinates are harmonic coordinates for a conformal metric whose scalar curvature vanishes. The conformal metric is achieved by solving a Yamabe problem locally for zero Yamabe constant. This is equivalent to finding a local positive solution $f>0$ to the linear equation $L_gf=0$. We however find that the formulation~\eqref{quotient_form} of conformal harmonic coordinates as quotients of solutions leads to new applications. The reason behind this is that harmonic coordinates for a zero scalar curvature metric are not conformally invariant. The quotient formulation also naturally generalizes isothermal coordinates. We show that conformal harmonic coordinates are characterized as a coordinate system where the contracted Christoffel symbols $\Gamma^a=g^{bc}\Gamma_{bc}^a$ and $\Gamma_a=g_{ab}\Gamma^b$ satisfy the coordinate condition
\[
 \Gamma_a=2\p_a\log f,
\]
where $f>0$ solves $L_gf=0$. A feature of this coordinate condition is that the function $f$ in the condition depends non-locally on $g$. This causes some additional considerations in some applications we present.  

 We give applications of conformal harmonic coordinates. 
 We first prove an up to boundary regularity result for conformal mappings.
\begin{Theorem}
Let $(M,g)$ and $(N,h)$ be $n$-dimensional, $n\geq 3$, Riemannian manifolds possibly with boundaries. Assume that the Riemannian metric tensors $g$ and $h$ are $C^{k+\alpha}$, $C^\infty$ or $C^\omega$ smooth, 
with $k\in \N$, $k\geq 2$ and $\alpha\in (0,1)$. 
%
Let $F: M \to N$ be $C^3$ smooth locally diffeomorphic conformal mapping,
\[
F^* h = c\,g \text{ in } M,
\]
where $c$ is a function in $M$. If $\p M\neq \emptyset$, we also assume $\p N\neq \emptyset$, that $g$ and $h$ are $C^{k+\alpha}$, $C^\infty$ or $C^\omega$ up to boundary, $F$ is $C^3$ up to boundary and that $F(\p M)\subset \p N$. Then $F$ and its local inverse are in $C^{k+1+\alpha}$, $C^\infty$ or $C^\omega$ respectively, up to boundary if $\p M \neq \emptyset$.


\end{Theorem}
After that we show a unique continuation result for conformal mappings. 
\begin{Theorem}
 Let $(M,g)$ and $(N,h)$ be connected $n$-dimensional Riemannian manifolds, $n\geq 3$. Assume that $F_1$ and $F_2$ are two $C^3$ locally diffeomorphic conformal mappings $M \to N$, 
 $F_j^*h=c_jg$, $j=1,2$, where $c_j$ are functions in $M$. If $M$ has a boundary, we also assume that $F_1$ and $F_2$ are $C^3$ up to $\p M$ and that $F_1(\p M)\subset \p N$.
 
 Assume either that $F_1=F_2$ on an open subset of $M$ or that $F_1$ and $F_2$ agree to second order on an open subset $\Gamma\subset \p M$, if $\p M \neq \emptyset$. Then 
 \[
  F_1=F_2 \text{ on } M.
 \]
 
\end{Theorem}
Here we regard that two mappings agree to second order on a set if the zeroth, first and second derivatives of the mappings are the same on the set.

The regularity result extends the results in the interior, see~\cite{LS1} and references therein. The latter result has also been studied in~\cite{Pa95}, but we give a new simple proof and include a unique continuation from the boundary result. Both of these results are based on the fact that conformal harmonic coordinates are preserved by conformal mappings.

We then continue by studying Bach, Fefferman-Graham obstruction, Cotton and Weyl conformal curvature tensors. These tensors are conformally invariant, which consequently fail to be elliptic unless both the coordinate and conformal invariances are properly fixed. 
On an even $n\geq 4$ dimensional manifold, the Fefferman-Graham obstruction tensors $\mathcal{O}=\mathcal{O}_{(n)}$ (or simply obstruction tensors) satisfy 
\begin{equation*}
 \mathcal{O}_{ab}=\frac{1}{n-3}\Delta^{n/2-2}B_{ab}+\text{lower order terms}.
\end{equation*}
If $(M,g)$ is a smooth Riemannian manifold, we show that by choosing conformal harmonic coordinates the conformal curvature tensors mentioned above become elliptic operators for the conformal metric $\abs{g}^{-1/n}g$. We call $\abs{g}^{-1/n}g$ the \emph{determinant normalized metric}. The choice of conformal harmonic coordinates and normalizing the determinant of the metric may be viewed as a local gauge condition that results in elliptic equations for conformal curvature tensors.  
In fact, we calculate in Proposition~\ref{symbol_for_Bach} that in this local gauge the principal symbol of the obstruction tensor $\mathcal{O}_{(n)}$ is just the scalar 
\[
 -\frac{1}{2(n-3)}\abs{\xi}^n.
\]
We prove the corresponding elliptic regularity result.
\begin{Theorem}
Let $(M,g)$ be an $n$-dimensional Riemannian manifold without boundary, and let $g \in C^r$, in some system of local coordinates. 
\begin{enumerate}
\item[(a)]
If $n \geq 4$, $r>2$, and if $W_{abc\phantom{d}}^{\phantom{abc}d}$ is in $C^{s}$, $C^\infty$ or $C^\omega$, for some $s> r-2$, $s\notin \Z$, in conformal harmonic coordinates, then $\abs{g}^{-1/n} g_{ab}$ is in $C^{s+2}$, $C^\infty$ or $C^\omega$ respectively in these coordinates. 

\item[(b)]
If $n = 3$, $r>2$, and if $C_{abc}$ is in $C^{s}$, $C^\infty$ or $C^\omega$, for some $s > r-3$, $s\notin \Z$, in conformal harmonic coordinates, then $\abs{g}^{-1/n} g_{ab}$ is in $C^{s+3}$, $C^\infty$ or $C^\omega$ respectively in these coordinates.

\item[(c)]
If $n \geq 4$ is even, $r > n-1$, and if $\abs{g}^{\frac{n-2}{2n}}\mathcal{O}_{ab}$ is in $C^s$, $C^\infty$ or $C^\omega$, for some $s > r-n$, $s\notin \Z$, in conformal harmonic coordinates, then $\abs{g}^{-1/n} g_{ab}$ is in $C^{s+n}$, $C^\infty$ or $C^\omega$ respectively in these coordinates. Here $\mathcal{O}=\mathcal{O}_{(n)}$.

\end{enumerate}
Especially, if a tensor $W$, $C$ or $\mathcal{O}$ vanishes, then $\abs{g}^{-1/n}g$ is real analytic $C^\omega$ in conformal harmonic coordinates.
\end{Theorem}
A similar result, though not including the $C^\omega$ case, was proven in a joint work~\cite{LS2} by the second mentioned author of this paper in a different gauge, which will be discussed below. Theorem~\ref{regularity_thm_interior} states that a Riemannian manifold whose Bach or obstruction tensor vanishes, is locally conformal to a real analytic manifold. We use real analyticity to prove the following results. 
We show that if on a Bach or obstruction flat manifold a point has a neighborhood conformal to the Euclidean space, then the same property holds for all points of the manifold.
\begin{Theorem}
  Let $(M,g)$ be an obstruction flat, $\mathcal{O}_{(n)}\equiv 0$, connected Riemannian manifold without boundary of even dimension $n\geq 4$ with $g\in C^r$, $r>n-1$. Assume that $W(g)=0$ on an open set of $M$. Then $(M,g)$ is locally conformally flat i.e.
  \[
   W(g)\equiv 0 \text{ on } M.
  \]
 \end{Theorem}
  Theorem~\ref{ucp_for_Bach_flat} can be regarded as a rigidity result for Weyl flat metrics. We also have that Bach and obstruction flat manifolds can be locally conformally embedded in $\R^{n(n+1)/2-1}$.
  \begin{Theorem}
 An obstruction flat, $\mathcal{O}_{(n)}\equiv 0$, Riemannian manifold $(M,g)$ of even dimension $\geq 4$ with $g\in C^{r}$, $r>n-1$, can be locally embedded in $\R^N$, $N=n(n+1)/2-1$, by a conformal mapping.
\end{Theorem}

Lastly we turn to discuss conformal harmonic coordinates on Lorentzian manifolds. On Lorentzian manifolds we call these coordinates \emph{conformal wave coordinates}. We show the existence of conformal wave coordinates and record coordinate formulas for conformal curvature tensors in these coordinates. Our final new result is a unique continuation of conformal mappings on Lorentzian manifolds.
\begin{Theorem}
 Let $(M,g)$ be a connected, oriented, time oriented, globally hyperbolic Lorentz manifold without boundary of dimension $n\geq 3$, and let $(N,h)$ be an $n$-dimensional Lorentz manifold without boundary. Let $S$ be a $C^\infty$ smooth spacelike Cauchy hypersurface in $M$. 
 Let $\Omega$ be a subset of $S$.
 
 Assume that $F_1$ and $F_2$ are two $C^\infty$ smooth locally diffeomorphic conformal mappings $M \to N$, 
 $F_1^*h=c_1 g$ and $F_2^*h=c_2 g$. Assume also that $F_1$ and $F_2$ agree to second order on $\Omega$. 
Then, for $p\in \mathcal{D}^+(\Omega)\cup \mathcal{D}^-(\Omega)$ holds 
 \[
  F_1(p)=F_2(p). 
 \]
 Here $\mathcal{D}^+(\Omega)$ and $\mathcal{D}^-(\Omega)$ are the future and past Cauchy developments of $\Omega$.
%
Especially if $\Omega=S$, then $F_1=F_2$ in $M$. 
\end{Theorem}

\subsection{Comparison with previous literature} 
The regularity of conformal mappings between Riemannian manifolds in the interior has been studied in~\cite{Ferrand, Re78, Iwaniec_thesis, Sh82}. Recent works on this subject include~\cite{LS1, JLS}, and the works~\cite{CCLO17} and~\cite{MT17} on subRiemannian and Finslerian manifolds respectively. 
The works~\cite{LS1, JLS} rely on constructing coordinates called $n$-harmonic coordinates on a Riemannian manifold of dimension $n$ that satisfy the $n$-Laplace equation $-\nabla^a(\abs{\nabla u}^{n-2}\nabla_au)=0$. The $n$-Laplace equation is nonlinear and degenerate elliptic. The work~\cite{CCLO17} applies coordinates on a subRiemannian manifold that solve a so called $Q$-Laplace equation, which 
is analogous to the $n$-Laplace equation.  
An additional assumption, which is satisfied on contact manifolds and related to the degeneracy of the $Q$-Laplace equation, is imposed in~\cite{CCLO17} to prove regularity of conformal mappings in general. Unique continuation of conformal mappings on Riemannian manifolds has also been studied in~\cite{Pa95}. A beneficial property of using conformal harmonic coordinates in studying regularity and unique continuation questions in conformal geometry seems to be that the involved equations are linear and elliptic (or hyperbolic in the Lorentzian setting). For example, boundary regularity of linear elliptic equations is standard. 

There are a number of results in conformal geometry that employ various gauges to obtain elliptic or parabolic regularity and existence results. We mention the works \cite{AcheViaclovsky, Anderson, Helliwell, Tian}, which study asymptotically locally Euclidean manifolds with obstruction flat metrics or boundary regularity of conformally compact Einstein metrics. These works involve constant scalar curvature to fix the conformal invariance and use harmonic coordinates for a conformal metric of constant scalar curvature. The works~\cite{LS2, JLS} uses $n$-harmonic coordinates and determinant normalized metric to have elliptic regularity results for conformal curvature tensors.

A conformal analogue of the Ricci flow in dimension $4$ is the Bach flow.
The work~\cite{BahuaudHelliwell} uses a version of the DeTurck trick to show that certain analogues of the Ricci flow, including flows involving the Bach and obstruction tensors, are locally well posed. The question of local wellposedness of the actual Bach flow is to the best of our knowledge still open, see e.g.~\cite{Ho18}. We also mention the works~\cite{GurskyViaclovsky1}, \cite{GurskyViaclovsky2} that study quadratic curvature functionals, where the corresponding 
Euler-Lagrange equations become elliptic via suitable gauge conditions. 

Let us give examples of gauge fixes in conformal geometry. For a $4$-dimensional Riemannian manifold whose Bach tensor vanishes, choosing a conformal metric with constant scalar curvature produces the equation
\[
 B_{ab}=\nabla^c \nabla_c R_{ab} + \text{lower order terms} = 0
\]
for the conformal metric. This equation becomes elliptic in harmonic coordinates for the conformal metric, see e.g.~\cite{Anderson, Helliwell, Tian}. In this paper, we fix the coordinate invariance by using conformal harmonic coordinates, which are indeed harmonic coordinates for a constant scalar curvature metric. However, we fix the conformal invariance by normalizing the determinant of the metric instead of using constant scalar curvature. Possible benefits of using our gauge (conformal harmonic coordinates + determinant one) over (constant scalar curvature + harmonic coordinates) are that the fix of the conformal invariance is local and explicit in the metric (used in~\cite{LLS}), and that in some applications fixing the determinant (the volume form) of the metric might be more natural. One example of the latter is the Bach flow, which preserves the volume form. The works~\cite{LS2, JLS} prove a theorem similar to Theorem~\ref{regularity_thm_interior} in a gauge, which combines the use of $n$-harmonic coordinates with determinant normalized metric. 
A possible benefit of our gauge (conformal harmonic coordinates + determinant one) over the gauge ($n$-harmonic coordinates + determinant one), is that the involved equations are linear and elliptic instead of nonlinear and degenerate elliptic. Also the principal symbol of the obstruction tensor in the gauge of~\cite{LS2,JLS} is not a scalar as it is in our gauge, which can be useful for example in studying unique continuation. In addition, there is no theory of $n$-Laplace equation in Lorentzian geometry to the best of our knowledge. 


Lastly we mention the recent interest in physics in conformal gravity~\cite{Ma11, Ma12}. This is a theory of Lorentz manifolds governed by the Bach tensor. In a vacuum, the field equations are $B_{ab}=0$. In the leading order, the equations $B_{ab}=0$ in conformal wave coordinates are just the square of the wave operator $\square^{\,2}$ operating on the determinant normalized metric. In conformal wave coordinates, the contracted symbols $\Gamma_a$ for the determinant normalized metric satisfy $\Delta \Gamma_a=0$ in the leading order. This condition is similar to the condition of harmonic coordinates, which have been useful in the study of the Einstein's general relativity. Because of these formulas, our gauge (conformal wave coordinates + determinant one) has potential to be a useful gauge in conformal gravity.  Section~\ref{sec_lorentzian_manifolds} for the formulas. 

As a summary, conformal harmonic coordinates, used together with determinant normalization in some applications, seem to combine some of the advantageous properties of different previous gauges. Conformal harmonic coordinates are a close conformal analogue of harmonic coordinates and other Riemannian applications of harmonic coordinates might have conformal analogues by using conformal harmonic coordinates. 
A notable disadvantage of conformal harmonic coordinates is that a notion of scalar curvature is needed. Since scalar curvature contains derivatives of the metric up to second order, the works~\cite{LS2, JLS} can prove results for metrics of lower regularity than we can do in this paper. On this matter, we mention a distributional formulation of the scalar curvature studied in the recent paper~\cite{LF15}.

\subsection{Organization of the paper}
The paper is structured as follows. In Section~\ref{sec:conformal_harmonic_coordinates} we prove the existence and regularity of conformal harmonic coordinates and characterize them. In Section~\ref{applications_of_conf_harm} we apply conformal harmonic coordinates in Riemannian geometry by first studying regularity and unique continuation properties of conformal mappings. In Section~\ref{Bach_formulas} we study conformal curvature tensors and regularity of Riemannian metrics satisfying conformal curvature equations. In Section~\ref{Z-coord_lorentz} we discuss conformal wave coordinates and prove a unique continuation results for conformal mappings on a Lorentzian manifold.
The appendix contains a list of formulas in Riemannian geometry and an existence result for solutions to the conformal Laplacian equation with some prescribed properties at a single point. 

\subsection{Acknowledgments}
The authors were supported by the Finnish Centre of Excellence in Inverse Modelling and Imaging (Academy of Finland grant 284715).
The authors wish to thank Vesa Julin and Mikko Salo for helpful discussions. We also wish to thank J\"urgen Jost for providing additional references regarding the use of the usual harmonic coordinates.
\section{Conformal harmonic coordinates}\label{sec:conformal_harmonic_coordinates}
In this section we introduce conformal harmonic coordinates.
We prove that if a Riemannian manifold $(M,g)$ has a $C^r$ regular metric tensor $g$, $r>2$, then on a neighborhood of any point on the manifold $M$ exists a conformal harmonic coordinate system. After proving the existence, we derive characteristic properties of conformal harmonic coordinates including a coordinate condition for the Christoffel symbols. 

The manifold $M$ may have a boundary $\p M$ in some parts of the text. We always assume that dimensions of the manifolds we consider are at least three,
\[
 n=\dim(M)\geq 3.
\]
If $M$ has a boundary, we denote the interior points of $M$ by $\Mi$. We call a coordinate system $U$ on a neighborhood $\Omega$ of a boundary point a \emph{boundary coordinate system} if
\[
U(\Omega)\subset \{x\in \R^n: x^n\geq 0\} \text{ and } U(\Omega\cap \p M)\subset \{x\in \R^n: x^n=0\}.
\]
If the Riemannian metric is $C^r$, we assume that the transition functions of the atlas of the manifold are $C^{r+1}$. Especially, if the manifold has a boundary, we assume the boundary is $C^{r+1}$ and that the Riemannian metric is $C^r$ up to $\{x^n=0\}$ in any boundary chart of the atlas. If $r=\infty$ or $r=\omega$ we assume that the transition functions are $C^\infty$ smooth or real analytic $C^\omega$. We refer to~\cite{Le13} for more details about manifolds with boundary. 

Conformal harmonic coordinates are defined as follows~\cite{LLS}.
\begin{Definition}\label{def:Z_coords}
 Let $(M,g)$ be a Riemannian manifold, possibly with boundary, and where $g$ is in $C^r$ with $r>2$. Local coordinates $(Z^1,\ldots,Z^n)$ on a neighborhood $\Omega$ of a point $p\in M$ are called conformal harmonic coordinates if the components $Z^k$, $k=1,\ldots,n$, are functions on $\Omega$ of the form
\begin{equation}\label{Z_form1}
Z^k=\frac{f^k}{f},
\end{equation}
where $f^k$ and $f>0$ are any functions that satisfy
\[
 L_gf^k=L_gf=0 \text{ on } \Omega.
\]
If $p\in \p M$, we additionally require that $(Z^1,\ldots,Z^n)$ is a boundary coordinate system. 
\end{Definition}
Conformal harmonic coordinates were invented in~\cite{LLS}. In that paper the coordinates were called $Z$-coordinates and their existence were shown for neighborhoods of boundary points and under an additional assumption. 
We prove the existence also for the interior points and without additional assumptions. We also consider the cases of H\"older regular and real analytic metrics. In proving the existence, we will use Lemma~\ref{prescribed_sols} in the appendix to have solutions to the conformal Laplace equation with prescribed values and gradients at a single given point.

\begin{Theorem}[Existence in the interior]\label{Z-coord}
Let $(M,g)$ be a Riemannian manifold without boundary whose metric is in $C^r$, with $r>2$. Let $p\in M$. 


\noindent \textbf{(1)} There exist conformal harmonic coordinates 
on a neighborhood of $p$. 

\noindent \textbf{(2)} If $r=k+\alpha$, $k\in \N$, $k\geq 2$, and $\alpha\in (0,1)$, all conformal harmonic

coordinates defined on a neighborhood of $p$ have $C^{r+1}$ regularity. If 

additionally $g\in C^\infty$ or $g\in C^\omega$, then all conformal harmonic coordinates 

are $C^\infty$ or $C^\omega$ respectively.

\noindent \textbf{(3)} There exist conformal harmonic coordinates so that the coordinate 

representation of the metric in these coordinates satisfies $g_{jk}(p)=\delta_{jk}$. 

%
%
%
%
\end{Theorem}
\begin{proof}
 Let $p\in M$. Let $(x^k)$ be a coordinate system on a neighborhood of $p$ where the coordinate representation of the Riemannian metric is in $C^r$. We may assume that $x(p)=0$. By Lemma~\ref{prescribed_sols}, there is $\eps>0$ and functions $f^k\in C^r$, $k=1,\ldots,n$, so that
 \begin{equation}\label{fks}
  L_gf^k=0 \mbox{ on } B(0,\eps), \mbox{ with } df^k(0)=dx^k \text{ and } f^k(0)=0.
 \end{equation}
By the same lemma, there is also a positive function $f\in C^r$ such that $L_gf=0$ on $B(0,\eps)$ and $f(0)=1$.
 
 Let us denote $F=(f^1,\ldots, f^n)$. The Jacobian matrix of the mapping $Z=(Z^1,\ldots, Z^n)$, where $Z^k=f^k/f$, reads
 \begin{equation}\label{differential_of_Z}
 DZ=\frac{1}{f}DF -\frac{1}{f^2}F\otimes d f.
 \end{equation}
 Here $F\otimes d f$ is a matrix field with components $(F\otimes d f)^k_j=f^k\p_jf$. By the equations~\eqref{fks} and~\eqref{differential_of_Z} and the condition $f(0)=1$, we have that 
 \[
  DZ(0)=I_{n\times n}.
 \]
 Thus $Z=(Z^1,\ldots, Z^n)$ is a coordinate system on a neighborhood of $p$ by the inverse function theorem.
 
 By construction, the coordinate system $Z$ is only $C^{r}$ regular and we still need to prove the higher $C^{r+1}$ regularity. For this we note that the metric 
 \[
  \tilde{g}=f^{p-2}g\in C^r,
 \]
 where 
 \[
p=\frac{2n}{n-2},
\]
has zero scalar curvature due to the formula~\eqref{scalar_curvature_scaled2}, 
 \begin{equation*}
 R(\tilde{g})=f^{1-p}\left(4\frac{n-1}{n-2}\Delta_gf+R(g)f\right)=4\frac{n-1}{n-2}f^{1-p}L_gf=0.
 \end{equation*}
 Using~\eqref{scalar_curvature_scaled2} and the conformal invariance~\eqref{conformal_invariance_of_L} of the conformal Laplacian, we have that
 \begin{equation}\label{additional_regularity}
 \Delta_{f^{p-2}g}\frac{f^k}{f}=L_{f^{p-2}g}\frac{f^k}{f}=f^{1-p}L_gf^k=0.
 \end{equation}
 Thus each $f^k/f$ satisfies an elliptic equation in divergence form with $C^{r}$ coefficients. It follows that $Z\in C^{r+1}$ by elliptic regularity (see e.g.~\cite[Theorem 6.17]{GT}). The equation ~\eqref{additional_regularity} also shows that any conformal harmonic coordinate system has $C^{r+1}$ regularity.
If $g$ is real analytic, the coordinates $Z$ are $C^\omega$ by real analytic elliptic regularity 
(see e.g.~\cite[Appendix J]{Besse}). We have proven the claims $(1)$ and $(2)$ about the existence and regularity of conformal harmonic coordinates.
 
 To prove the claim $(3)$, we choose the solution $f$ as before, but apply Lemma~\ref{prescribed_sols} to have solutions $f^k$ that satisfy
 \[
  \p_jf^k(0)=A_j^k,
 \]
 and $f^k(0)=0$, where $A$ is an invertible matrix that satisfies $g(0)=A^TA$. 
\end{proof}

\begin{Theorem}[Existence near boundary]\label{Z-coord_boundary}
Let $(M,g)$ be a Riemannian manifold with boundary whose metric is in $C^r$, with $r>2$. Let $p\in M$.
\noindent \textbf{(1)} There exist conformal harmonic coordinates 
on a neighborhood of $p$. 

\noindent \textbf{(2)} If $p\in \p M$, if $r=k+\alpha$, $k\in \N$, $k\geq 2$ and $\alpha\in (0,1)$, and if $g$ is in $C^r$, 

$C^\infty$ or $C^\omega$, then all conformal harmonic coordinates, which are defined 

on a neighborhood of $p$, and whose restrictions to the boundary are $C^{r+1}$, 

$C^\infty$ or $C^\omega$, are in $C^{r+1}$, $C^\infty$ or $C^\omega$ up to boundary respectively.

\noindent \textbf{(3)}  If $(u^1,\ldots,u^{n})$ is $C^{r+1}$ boundary coordinate system of $M$, then there exist \indent conformal harmonic coordinates whose restriction to the boundary equals \indent $(u^1,\ldots,u^{n-1},0)$. 
\end{Theorem}
\begin{proof}
 Existence in the interior was proven in Theorem~\ref{Z-coord}. Let $p\in \p M$. Let $U=(u^1,\ldots,u^n)$ be a boundary coordinate system on a neighborhood $\Omega$ of $p$, where the coordinate representation of the metric $g$ is of class $C^r$ up to the boundary of the set $U(\Omega)\cap \{x^n\geq 0\}\subset \R^n$. By Lemma~\ref{prescribed_sols}, there is $\eps>0$ and a function $f>0$ so that $L_gf=0$ on $D_\eps:=\overline{B}(0,\eps)\cap \{x^n\geq 0\}$ and $f=1$ on $\Gamma_\eps:=D_\eps\cap \{x^n=0\}$. 
 We scale the metric $g$ as
 \[
  \tilde{g}=f^{p-2}g,
 \]
where $p=\frac{2n}{n-2}$, so that by~\eqref{scalar_curvature_scaled2} we have that
 \[
  L_{\tilde{g}}=\Delta_{\tilde{g}}.
 \]
 Since $\Delta_{\tilde{g}}$ has no zeroth order term, we may solve for $k\in \{1,\ldots, n-1\}$ the boundary value problems (see~\cite[Theorem 8.9]{GT})
\begin{align*}
 L_{\tg}v^k&=0, \mbox{ on } D_\eps\subset \R^n, \\
 v^k &=x^k \mbox{ on } \p D_\eps\subset \R^{n-1} \nonumber.
\end{align*}
We remind that we work in the coordinates $U$, which for instance means that $x^k=u^{k}\circ U^{-1}$. By the conformal invariance~\eqref{conformal_invariance_of_L} of $L_g$ we have that the function
\begin{equation*}\label{v_is_Z_1}
 f^k=fv^k
\end{equation*}
solves $L_gf^k=0$ on $D_\eps$ and $f^k=x^k$ on $\Gamma_\eps$. 

Let $f^n$ be a solution to $L_gf^n=0$ in $D_\eps$ with $f^n= 0$ on $\Gamma_\eps$ and $df^n(0)=dx^n$. The existence of $f^n$ is guaranteed by Lemma~\ref{prescribed_sols}. Let us define
\begin{equation*}\label{v_is_Z_2}
 Z^k=\frac{f^k}{f}, \quad k=1,\ldots,n.
\end{equation*}
Then $DZ(0)=I_{n\times n}$ by the equation~\eqref{differential_of_Z}. Thus $Z=(Z^1,\ldots,Z^n)$ is a conformal harmonic coordinate system on a neighborhood of $p$. It is also a boundary coordinate system and it agrees with $(u^1,\ldots,u^{n-1},0)$ on the boundary since $Z^k|_{\Gamma_\eps}=x^k$. 
The claim about $C^{r+1}$, $C^\infty$ or $C^\omega$ regularity follows from the equation~\eqref{additional_regularity} and boundary regularity results for elliptic equations, see e.g.~\cite[Theorem 6.19]{GT} and~\cite[Theorem A]{MN57}. 
\end{proof}



 We next observe that conformal harmonic coordinates are harmonic coordinates for a conformal metric that has zero scalar curvature. Finding the conformal scaling can be thought as solving the Yamabe problem
 \[
  L_gf=\Delta_g + \frac{n-2}{4(n-1)}R(g)=0 
 \]
 locally for a zero Yamabe constant. If $f>0$ solves the above, then $R(f^{p-2}g)=f^{1-p}\left(4\frac{n-1}{n-2}\Delta_gf+R(g)f\right)=0$. Here $p=\frac{2n}{n-2}$ as usual. By Lemma~\ref{prescribed_sols} such an $f$ can always be found on a sufficiently small domain.
 
 Recall that a coordinate system $U=(u^1,\ldots,u^n)$ is harmonic with respect to a Riemannian metric $g$ if all the coordinate functions $u^k$ are harmonic:
 \begin{equation}\label{harmonic_coordinates}
  \Delta_gu^k=0, \quad k=1,\ldots,n.
 \end{equation}
We refer to~\cite{DK} for more details on harmonic coordinates. In the rest of this section the Riemannian metric $g$ is assumed to be $C^r$, $r>2$, without further notice.  

\begin{Proposition}\label{equivalence}
 Let $(M,g)$ be a Riemannian manifold possibly with boundary. Let $Z=(Z^1,\ldots,Z^n)$ be coordinates on an open set $\Omega\subset M$. 
 Then $Z$ are conformal harmonic coordinates on $\Omega$ if and only if there is a function $f>0$ such that $Z$ are harmonic coordinates with respect to the conformal metric $f^{p-2}g$ whose scalar curvature vanishes, $R(f^{p-2}g)=0$. 
\end{Proposition}
\begin{proof} 
 Assume that $Z=(Z^1,\ldots,Z^n)$ are conformal harmonic coordinates on an open set $\Omega$. By definition, we have that
 \[
 Z^k=\frac{f^k}{f}, \quad k=1,\ldots, n, 
 \]
 where $f^k$ and $f$ solve $L_gf^k=L_gf=0$, with $f>0$.
 The scalar curvature transforms under the conformal scaling $\tilde{g}=f^{p-2}g$, $p=\frac{2n}{n-2}$, 
 by the formula~\eqref{scalar_curvature_scaled2}
 \begin{equation}\label{yamabe_scaling}
 R(\tilde{g})=f^{1-p}\left(4\frac{n-1}{n-2}\Delta_gf+R(g)f\right).
 \end{equation}
 Thus we have $R(f^{p-2}g)=0$, since the function $f$ solves
 \[
L_gf=\Delta_gf +\frac{n-2}{4(n-1)} R(g)f=0.
 \]
 
 We prove that $(Z^1,\ldots,Z^n)$ is a harmonic coordinate system for the conformal metric $\tilde{g}=f^{p-2}g$. To prove it, we need to show that each $Z^k$, $k=1,\ldots,n$, satisfies $\Delta_{\tilde{g}}Z^k=0$. The conformal Laplacian transforms under a conformal scaling by the equation~\eqref{conformal_invariance_of_L} as
 \begin{equation}\label{L_conformal_trans}
 L_{\tilde{g}}(f^{-1}u)=f^{1-p}L_gu.
 \end{equation} 
 We calculate
 \[
 \Delta_{\tilde{g}}Z^k=L_{\tilde{g}}\frac{f^k}{f}=f^{1-p}L_gf^k=0.
 \]
 Here in the first equality we have used the fact that $R(\tilde{g})=0$. Thus $(Z^1,\ldots,Z^n)$ are harmonic coordinates for the conformal metric $\tilde{g}=f^{p-2}g$ with $R(\tilde{g})=0$ on $\Omega$.
 
 To prove the opposite implication of the propositions claim, assume that there is a positive function $f$ on $\Omega$ such that
 \[
 R(f^{p-2}g)=0,
 \]
 and that $Z^k$ solves
 \[
 \Delta_{f^{p-2}g}Z^k=0, \quad k=1,\ldots,n.
 \]
 Let us write
 \begin{equation}\label{num_denom}
 Z^k=\frac{fZ^k}{f}
 \end{equation}
 and denote $\tilde{g}=f^{p-2}g$ as usual.
 We need to show that both the numerator and the denominator satisfy the conformal Laplace equation for the metric $g$. By the equation~\eqref{L_conformal_trans} and the fact that $R(\tilde{g})=0$, it follows that
 \[
 L_g(fZ^k)=f^{p-1}L_{\tilde{g}}(f^{-1}fZ^k)=f^{p-1}\Delta_{\tilde{g}}Z^k=0.
 \]
 For the function $f$ in the denominator of~\eqref{num_denom}, we have similarly $L_g(f)=f^{p-1}\Delta_{\tilde{g}}(1)=0$.
 Thus $(Z^1,\ldots,Z^n)$ are conformal harmonic coordinates with respect to the metric $g$.
\end{proof}

%
We will next show that 
a necessary and sufficient condition for coordinates to be conformal harmonic coordinates is the condition 
\[
 \Gamma_a=2\p_a \log f
\]
for the contracted Christoffel symbols $\Gamma_a$.  
Here $f$ is a positive function satisfying $L_gf=0$. Contracted Christoffel symbols are defined as $\Gamma^a=g^{bc}\Gamma_{bc}^a$ and $\Gamma_a=g_{ab}\Gamma^b$, see Appendix~\eqref{list_of_formulas}.
\begin{Proposition}\label{gauge_cond_of_Z}
Let $(M,g)$ be a Riemannian manifold possibly with boundary. Let $Z=(Z^1,\ldots,Z^n)$ be coordinates on an open set $\Omega\subset M$. 
 Then $Z$ are conformal harmonic coordinates on $\Omega$ if and only if
the contracted Christoffel symbols $\Gamma_a$, $a=1,\ldots, n$, in the coordinates $Z$ satisfy
 \begin{equation}\label{prop_gauge_cond}
 \Gamma_a=2\p_a \log f
 \end{equation}
 for some positive function $f$ on $\Omega$ that satisfies $L_gf=0$.
\end{Proposition}
\begin{proof}
Assume first that $(Z^1,\ldots,Z^n)$ are conformal harmonic coordinates on $\Omega$. Denote as usual $\tilde{g}=f^{p-2}g$, where the function $f>0$ solves $L_gf=0$ on $\Omega$. By Proposition~\ref{equivalence}, we have that $Z$ are harmonic coordinates for the metric $\tilde{g}$. Coordinates $(Z^k)$ are harmonic coordinates for the metric $\tilde{g}$ if and only if $\Gamma_a(\tilde{g})=0$, $a=1,\ldots,n$. This is well known~\cite{DK} and it follows from the calculation 
\begin{equation}\label{suf_and_nec_for_harm}
0=\Delta_{\tilde{g}}Z^a=-\tilde{g}^{bc}(\p_b\p_cZ^a-\Gamma_{bc}^d(\tilde{g})\p_dZ^a)=\tilde{g}^{bc}\Gamma_{bc}^d(\tilde{g})\delta_d^a=\Gamma^a(\tilde{g}).
\end{equation}

By lowering the index of $\Gamma^a$, we have in the coordinates $Z$ that
 \begin{equation}\label{gamma_scaled}
 \Gamma_a(\tilde{g})=\Gamma_a(f^{p-2}g)=0.
 \end{equation}
The contracted Christoffel symbols can be written by the formula~\eqref{stnd_contracted} as
\[
\Gamma_a(g) = g^{bc} \partial_b g_{ac} - \frac{1}{2} \partial_a(\log\abs{g})
\]
If $c$ is positive function, we have that
\begin{align}\label{conforal_scaling_of_contracted_christoffel}
 \Gamma_a(cg)&=(cg)^{bc} \partial_b (cg)_{ac} - \frac{1}{2} \partial_a(\log\abs{cg}) \nonumber \\
 &=g^{bc} \partial_b g_{ac}+\frac{1}{c}\p_a\log c -\frac{1}{2} \partial_a(\log\abs{g}+\partial_a\log c^n) \nonumber \\
 &=\Gamma_a(g)-\frac{n-2}{2}\p_a\log c.
\end{align}
Applying this formula for $c=f^{p-2}$ yields
\begin{equation}\label{scaled_gamma_pm2}
 \Gamma_a(f^{p-2}g)=\Gamma_a(g)-\frac{n-2}{2}(p-2)\p_a\log f=\Gamma_a(g)-2\p_a\log f,
\end{equation}
which in combination with~\eqref{gamma_scaled} yields~\eqref{prop_gauge_cond}.

 To see the opposite implication, assume that in a system $Z=(Z^1,\ldots, Z^n)$ of coordinates holds $\Gamma_a=2\p_a\log f$, $a=1,\ldots,n$, and where $f$ satisfies $L_gf=0$. The equation~\eqref{scaled_gamma_pm2} shows that $\Gamma_a(f^{p-2}g)=0$.
 Since coordinates are harmonic if and only if the contracted Christoffel symbols vanish, we have that 
 $(Z^1,\ldots, Z^n)$ are harmonic coordinates for the metric $f^{p-2}g$. 
 As we have by now seen many times, the conformal metric $f^{p-2}g$ has zero scalar curvature by the formula
 \[
 R(f^{p-2}g)=f^{1-p}\left(4\frac{n-1}{n-2}\Delta_gf+R(g)f\right).
 \]
 Thus $Z$ are conformal harmonic coordinates by Proposition~\ref{equivalence}.
\end{proof}

The next proposition shows that the coordinate condition~\eqref{prop_gauge_cond} for the contracted Christoffel symbols is in a specified sense conformally invariant.
\begin{Proposition}\label{Z_cond_conf_inv}
 Let $(M,g)$ be a Riemannian manifold possibly with boundary. Let $(Z^1,\ldots,Z^n)$ be conformal harmonic coordinates on an open set $\Omega\subset M$.
 Let $c\in C^r(M)$, $r>2$, be a positive function. The contracted Christoffel symbols for the conformally scaled metric $cg$ in the coordinates $(Z^1,\ldots,Z^n)$ satisfy
 \begin{equation}\label{scaled_Cs}
 \Gamma_a(cg)=2\p_a \log f_{cg} 
 \end{equation}
 for $a=1,\ldots,n$,  where 
 \begin{equation}\label{conformal_f}
f_{cg}=c^{-\frac{n-2}{4}}f_g  
 \end{equation}
is a positive function that satisfies
 \begin{equation}\label{scaled_eq}
 L_{cg}f_{cg}=0.
 \end{equation}
\end{Proposition}
\begin{proof}
 Let $(Z^1,\ldots,Z^n)$ be conformal harmonic coordinates on an open set $\Omega\subset M$. We have $\Gamma(g)_a=2\p_a \log f_{g}$ by Proposition~\ref{gauge_cond_of_Z}, where $L_gf_g=0$, $f_g>0$. 
The function
\[
f_{cg}=c^{-\frac{n-2}{4}}f_g
\]
satisfies
 \begin{equation}\label{f_gtilde_eq}
 L_{cg}f_{cg}=c^{-\frac{n+2}{4}}L_gf_g=0
 \end{equation}
by the conformal invariance~\eqref{conformal_invariance_of_L2} of the conformal Laplacian. The right hand side of~\eqref{scaled_Cs} now reads
\begin{equation}\label{trivial_identity}
2\p_a\log(c^{-\frac{n-2}{4}}f_g)=-\frac{n-2}{2}\p_a\log c +2\p_a\log f_g.
\end{equation}
By combining~\eqref{conforal_scaling_of_contracted_christoffel} and~\eqref{trivial_identity} and the coordinate condition $\Gamma(g)_a=2\p_a \log f_{g}$ of conformal harmonic coordinates,  we have that
\begin{align*}
\Gamma_a(cg)&=\Gamma_a(g)-\frac{n-2}{2}\p_a\log c=\Gamma_a(g)+2\p_a\log(c^{-\frac{n-2}{4}}f_g)-2\p_a\log f_g \\
&=2\p_a\log f_{cg},
\end{align*}
where $L_{cg}f_{cg}=0$ by~\eqref{f_gtilde_eq}. This proves the claim.
\end{proof}
A consequence of the results above is that conformal harmonic coordinates, and especially their smoothness, are conformally invariant. 
\begin{Proposition}\label{Z_coords_conformally_invariant}
 Let $(M,g)$ be a Riemannian manifold possibly with boundary. Assume that $c\in C^r(M)$, $r>2$, is a positive function. Let $Z$ be conformal harmonic coordinates defined with respect to $g$. Then $Z$ are conformal harmonic coordinates for the metric $cg$.
 
 If $g$ is in $C^{k+\alpha}$, $C^\infty$ or $C^\omega$, 
 then all conformal harmonic coordinates  for the conformal metric $cg$ are $C^{k+1+\alpha}$, $C^\infty$ or $C^\omega$ respectively. Here we assume $k\in \N$, $k\geq 2$ and $\alpha\in (0,1)$. (In the case of conformal harmonic coordinates on a neighborhood of a boundary point, we also assume that the restrictions of the coordinates to the boundary are $C^{k+1+\alpha}$, $C^\infty$ or $C^\omega$.) 
\end{Proposition}
\begin{proof}
Let $Z$ be conformal harmonic coordinates defined with respect to $g$. By Proposition~\ref{Z_cond_conf_inv}, we have in the coordinates $Z$ that $\Gamma_a(cg)=2\p_a \log f_{cg}$,
where 
$L_{cg}f_{cg}=0$. Thus by Proposition~\ref{gauge_cond_of_Z} we have that $Z$ are conformal harmonic coordinate system for the metric $cg$ as claimed. (Alternatively one can write $Z^k=f^k/f=(c^{-\frac{n-2}{4}}f^k)/(c^{-\frac{n-2}{4}}f)$
and notice that $L_{cg}\big(c^{-\frac{n-2}{4}}f^k\big)=L_{cg}\big(c^{-\frac{n-2}{4}}f\big)=0$ to see that the definition of conformal harmonic coordinates is satisfied.)
The claim about regularity follows from part (2) of Theorem~\ref{Z-coord}, or from part (2) of Theorem~\ref{Z-coord_boundary} if $M$ has boundary.
\end{proof}

\color{black}

We end this section by noticing that conformal harmonic coordinates are a generalization of isothermal coordinates. 
Recall that coordinates $(x^1,\ldots,x^n)$ on an open set are isothermal if the condition
\begin{equation}\label{isotherm}
 g_{ab}(x)=c(x)\,\delta_{ab}
\end{equation}
holds in the coordinates $(x^a)$ for some positive function $c$. We consider here only manifolds without boundary for simplicity.
\begin{Proposition}\label{prop:isothermal}
 Let $(M,g)$ be a Riemannian manifold without boundary. 
 Assume that coordinates $(x^1,\ldots,x^n)$ are isothermal in the sense of~\eqref{isotherm}, with $c\in C^r$, $r>2$. Then $(x^1,\ldots,x^n)$ are conformal harmonic coordinates.
\end{Proposition}
\begin{proof}
 Assume that $(x^1,\ldots,x^n)$ are isothermal coordinates, $g_{ab}=c\,\delta_{ab}$, on an open subset of $M$. 
By Proposition~\ref{gauge_cond_of_Z}, we only need to show that the coordinate condition 
 \[
 \Gamma_k(g)=2\p_k\log f
 \]
 of conformal harmonic coordinates holds in the coordinates $(x^k)$ for some positive function $f$ satisfying $L_gf=0$.
 Denote by $e$ the Euclidean metric with components $e_{ab}=\delta_{ab}$. By the formula~\eqref{conforal_scaling_of_contracted_christoffel} we have in the isothermal coordinates $(x^k)$ that
 \begin{align*}
 \Gamma_a(g)&=\Gamma_a(ce)=\Gamma_a(e)-\frac{n-2}{2}\p_a\log(c)=2\p_a\log c^{\frac{2-n}{4}}.
 \end{align*}
 We set $f=c^{\frac{2-n}{4}}$. Since the scalar curvature of the Euclidean metric $e$ vanishes, we have by~\eqref{conformal_invariance_of_L2} that
 \[
 L_gf=L_{ce}f=c^{-\frac{n+2}{4}}L_{e}(c^{\frac{n-2}{4}}c^{\frac{2-n}{4}})=c^{-\frac{n+2}{4}}\Delta_{e}(1)=0.
 \]
 By Proposition~\ref{gauge_cond_of_Z}, the coordinates $(x^k)$ are conformal harmonic coordinates. (The claim alternatively follows from Proposition~\ref{equivalence} by noticing that if $
  \tilde{g}=f^{p-2}g
  =c^{-1}g=e$,
 then trivially $R(\tilde{g})=0$ and the condition of harmonic coordinates $\Gamma_a(\tilde{g})=0$ also holds.)\qedhere
\end{proof}
For smooth enough metrics a necessary and sufficient condition for the existence of isothermal coordinates is that the Weyl and Cotton curvature tensor vanishes in dimension $n\geq 4$ and $n=3$ respectively. (For the required smoothness, see~\cite{LS2}.) Since conformal harmonic coordinates always exist, we conclude that conformal harmonic coordinates are not always isothermal.

\section{Applications of conformal harmonic coordinates}\label{applications_of_conf_harm}
In this section we give applications of conformal harmonic coordinates. We begin with a regularity result for conformal mappings, which contains a boundary regularity result. Then we prove a unique continuation result for conformal mappings. These are Theorems~\ref{reg_of_conformal_maps} and~\ref{ucp_for_conformal}.

After the first applications we move on to study conformal curvature tensors. We show that conformal curvature tensors such as the Bach and Fefferman-Graham obstruction tensors can be regarded as elliptic operators for a determinant normalized metric in conformal harmonic coordinates. 

We prove an elliptic regularity result for conformal curvature tensors in conformal harmonic coordinates in Theorem~\ref{regularity_thm_interior}. This shows for example that obstruction flat manifold are locally real analytic up to a conformal factor. We use the real analyticity to show that if on a Bach or obstruction flat manifold a point has a neighborhood conformal to the Euclidean space, then the same property holds for all points of the manifold. We also give a local conformal embedding result of obstruction flat manifolds into Euclidean space.
%
These are Theorems~\ref{ucp_for_Bach_flat} and~\ref{conformal_Cartan_Janet}.

In this section the Riemannian metrics are assumed to be in $C^r$, $r>2$, up to boundary when applicable, without further notice. As before, if $g\in C^r$, we assume that the atlas of the manifold is $C^{r+1}$. We also assume that dimensions of the manifolds are $\geq 3$.

\subsection{Regularity and unique continuation of conformal mappings}\label{reg_ucp_for_conf}
The regularity result for conformal mappings follows from the fact that conformal harmonic coordinates are preserved by conformal mappings. 
\begin{Lemma}\label{conf_preserves_Z}
 Let $(M,g)$ and $(N,h)$ be Riemannian $n$-dimensional manifolds possibly with boundaries, $n\geq 3$. Assume 
 that $F:(M,g)\to (N,h)$ is $C^{3}$ smooth locally diffeomorphic conformal mapping, $F^*h=c\,g$,
 where $c$ is a function on $M$. If $\p M\neq \emptyset$, we also assume that $F$ is $C^3$ up to boundary and that $F(\p M)\subset \p N$.
 Then $F$ pulls back any function of the form
 \[
  \frac{u_2}{v_2}, 
 \]
 where $u_2$ and $v_2> 0$ are locally defined functions on $N$ satisfying $L_hu_2=L_hv_2=0$,
  to a function of the form
 \[
  \frac{u_1}{v_1},
 \]
 where $L_gu_1=L_gv_1=0$ and  $v_1>0$.
  Especially $F$ pulls back conformal harmonic coordinates on $(N,h)$ to conformal harmonic coordinates on $(M,g)$.
\end{Lemma}
\begin{proof}
 Let $u_2$ and $v_2>0$ satisfy $L_hu_2=L_hv_2=0$. We write 
 \[
  F^*\left(\frac{u_2}{v_2}\right)=\frac{F^*u_2}{F^*v_2}=\frac{c^{\frac{n-2}{4}}F^*u_2}{c^{\frac{n-2}{4}}F^*v_2}.
 \]
By applying the conformal Laplacian on $(M,g)$ to the numerator, we have by the conformal invariance~\eqref{conformal_invariance_of_L2} that 
\begin{align}\label{scaled_pull_back_calc}
 L_g(c^{\frac{n-2}{4}}F^*u_2)&=L_{\frac{1}{c}F^*h}(c^{\frac{n-2}{4}}F^*u_2)=c^{\frac{n+2}{4}}L_{F^*h}F^*u_2=c^{\frac{n+2}{4}}F^*(L_hu_2)=0.
 \end{align}
The assumption that $F$ is $C^{3}$ was used here to justify the calculations. Likewise we have that $L_g(c^{\frac{n-2}{4}}F^*v_2)=0$. 
Thus
 \[
 F^*\left(\frac{u_2}{v_2}\right)=\frac{u_1}{v_1}
 \]
 where $L_gu_1=L_gv_1=0$ and $v_1>0$.
 
 Let $(Z^1,\ldots,Z^n)$ be conformal harmonic coordinates on an open subset of $N$. By the definition of the conformal harmonic coordinates, and by what we have already proven, we have that $F^*Z^k=F^*(f^k/f)=w^k/w$, where $L_gw^k=L_gw=0$, $w>0$. This proves the latter claim of the lemma. 
\end{proof}

\begin{Theorem}\label{reg_of_conformal_maps}
Let $(M,g)$ and $(N,h)$ be $n$-dimensional, $n\geq 3$, Riemannian manifolds possibly with boundaries. Assume that the Riemannian metric tensors $g$ and $h$ are $C^{k+\alpha}$, $C^\infty$ or $C^\omega$ smooth, 
with $k\in \N$, $k\geq 2$ and $\alpha\in (0,1)$. 
%
Let $F: M \to N$ be $C^3$ smooth locally diffeomorphic conformal mapping,
\[
F^* h = c\,g \text{ in } M,
\]
where $c$ is a function in $M$. If $\p M\neq \emptyset$, we also assume $\p N\neq \emptyset$, that $g$ and $h$ are $C^{k+\alpha}$, $C^\infty$ or $C^\omega$ up to boundary, $F$ is $C^3$ up to boundary and that $F(\p M)\subset \p N$. Then $F$ and its local inverse are in $C^{k+1+\alpha}$, $C^\infty$ or $C^\omega$ respectively, up to boundary if $\p M \neq \emptyset$.


\end{Theorem}

\begin{proof}
{\bf Interior points:}
 Let $p \in \Mi$ and let $Z=(Z^1,\ldots,Z^n)$ be conformal harmonic coordinates on a neighborhood $\Omega$ of $F(p)$ in $(N,h)$. 
 The mapping $Z:\Omega\to \R^n$ is a local $C^{r+1}$ diffeomorphism by Theorem \ref{Z-coord}. Define $W = F^*Z$. Then $W$ are also conformal harmonic coordinates on a neighborhood of $p$ by Lemma~\ref{conf_preserves_Z}. Thus $W$ and its local inverse are $C^{r+1}$ regular by Theorem \ref{Z-coord}. 
 Since $Z$ is an invertible mapping, we can express
 \begin{equation}\label{representation_for_conf}
 F=Z^{-1}\circ W.
 \end{equation}
 This mapping is $C^{r+1}$ regular as a composition of $C^{r+1}$ regular mappings. Since $W$ is locally invertible, it follows from~\eqref{representation_for_conf} that the local inverse of $F$ is also $C^{r+1}$.  
 
{\bf Boundary points:}
Let $p\in \p M$. Thus by assumption $F(p)\in \p N$. By Theorem~\ref{Z-coord_boundary}, on a neighborhood $\Omega$ of $F(p)$ in $(N,h)$ exist conformal harmonic coordinates $Z=(Z^1,\ldots,Z^n)$ which are $C^{r+1}$ regular up to the boundary. 
By Lemma~\ref{conf_preserves_Z} we have that $W=F^*Z$ are conformal harmonic coordinates on a neighborhood of $p$. 
Thus we may express $F$ in the form~\eqref{representation_for_conf}, which proves the claimed regularity in the interior. 

Next we show that $F$ restricted to the boundary is $C^{r+1}$ and then apply Theorem~\ref{Z-coord_boundary} part $(2)$ to show that $W$ is $C^{r+1}$ up to the boundary. To do that, we define 
\begin{align*}
 \overline{F}:(&\p M, g|_{\p M})\to  (\p N,h|_{\p N}), \\
 \overline{F}(q)&=F(q), \quad q\in \p M.
\end{align*}
As $F$ is conformal, we see that $\overline{F}$ is a conformal mapping between Riemannian manifolds without boundary that have $C^{r}$ regular Riemannian metrics. From the first part of this theorem, we have that $\overline{F}$ is $C^{r+1}$. Thus 
\[
 W|_{\p M\cap \Omega}=Z\circ F|_{\p M\cap \Omega}=Z\circ \overline{F} \in C^{r+1}(\p M\cap \Omega).
\]
Thus the restriction of $W$ to the boundary is $C^{r+1}$ and it follows from Theorem~\ref{Z-coord_boundary} part $(2)$ that $W$ are conformal harmonic coordinates, which have $C^{r+1}$ regularity up to boundary.  Finally we have that $F=Z^{-1}\circ W$ is a local diffeomorphism, which is $C^{r+1}$ regular up the boundary. 

The claims concerning $C^\infty$ or $C^\omega$ metrics follow similarly for both the interior and boundary points. 
\color{black}
\end{proof}
We remark that if $F$ in the theorem above is a homeomorphism, then the assumption $F(\p M)\subset \p N$ can be dropped. A homeomorphism maps boundary to boundary, but for a local homeomorphism this is not true (consider $z\mapsto z^2$ on a horseshoe shaped domain in $\mathbb{C}$).

We prove next a unique continuation result for conformal mappings, which has also been studied in~\cite{Pa95}. We say that two mappings $F_1$ and $F_2$ between $n$-dimensional manifolds $M$ and $N$ agree to second order on a set $\Xi \subset M$ if $F_1|_\Xi=F_2|_\Xi$ and if for each $p\in \Xi$ there are coordinates $(x^1,\ldots,x^n)$ and $(y^1,\ldots,y^n)$ on neighborhoods of the point $p$ and the point $F_1(p)=F_2(p)\in N$ respectively such that for $x\in \Xi$ holds
\[
 \p_{x_a}F_1^l(x)=\p_{x_a}F_2^l(x) \text{ and } \p_{x_a}\p_{x_b}F_1^l(x)=\p_{x_a}\p_{x_b}F_2^l(x)
\]
for all $a,b=1,\ldots,n$ and $l=1,\ldots,n$.


\begin{Theorem}\label{ucp_for_conformal}
 Let $(M,g)$ and $(N,h)$ be connected $n$-dimensional Riemannian manifolds, $n\geq 3$. Assume that $F_1$ and $F_2$ are two $C^3$ locally diffeomorphic conformal mappings $M \to N$, 
 $F_j^*h=c_jg$, $j=1,2$, where $c_j$ are functions in $M$. If $M$ has a boundary, we also assume that $F_1$ and $F_2$ are $C^3$ up to $\p M$ and that $F_1(\p M)\subset \p N$.
 
 Assume either that $F_1=F_2$ on an open subset of $M$ or that $F_1$ and $F_2$ agree to second order on an open subset $\Gamma\subset \p M$, if $\p M \neq \emptyset$. Then 
 \[
  F_1=F_2 \text{ on } M.
 \]
 
\end{Theorem}

\begin{proof}
Let us first assume that $F_1=F_2$ on an open subset of $M$. Let $B$ be the largest open subset of $M$ where $F_1=F_2$. We have that $B$ is nonempty and that
 \[
  c_1=g^{-1}F_1^*h=g^{-1}F_2^*h=c_2 \text{ on } B.
 \]
 Let $p\in \p B$. Since $F_1$ and $F_2$ are continuous, we have that $q:=F_1(p)=F_2(p)$.
Let $Z=(Z^1,\ldots,Z^n)=(f^1/f,\ldots,f^n/f)$ be conformal harmonic coordinates on a neighborhood $\Omega$ of $q$. We have for $k=1,\ldots, n$ and $j=1,2$ that  
\[
 F_j^*\left(\frac{f^k}{f}\right)=\frac{F_j^*f^k}{F_j^*f}=\frac{c_j^{(n-2)/4}F_j^*f^k}{c_j^{(n-2)/4}F_j^*f}.
\]
A calculation similar to the one in~\eqref{scaled_pull_back_calc} shows that 
\[
 L_g(c_j^{(n-2)/4}F_j^*f^k)=L_g(c_j^{(n-2)/4}F_j^*f)=0.
\]
We have on the nonempty open set $F_1^{-1}(\Omega)\cap F_2^{-1}(\Omega)\cap B$ that $c_1^{(n-2)/4}F_1^*f^k=c_2^{(n-2)/4}F_2^*f^k$. By the unique continuation principle of solutions to elliptic partial differential equations we have 
\[
 c_1^{(n-2)/4}F_1^*f^k=c_2^{(n-2)/4}F_2^*f^k \text{ and } c_1^{(n-2)/4}F_1^*f=c_2^{(n-2)/4}F_2^*f 
\]
on the open set $F_1^{-1}(\Omega)\cap F_2^{-1}(\Omega)$.
It follows that
\[
 F_1^*\left(\frac{f^k}{f}\right)=\frac{c_1^{(n-2)/4}F_1^*f^k}{c_1^{(n-2)/4}F_1^*f}=\frac{c_2^{(n-2)/4}F_2^*f^k}{c_2^{(n-2)/4}F_2^*f}=F_2^*\left(\frac{f^k}{f}\right).
\]
By denoting $W^k=F_1^*\left(\frac{f^k}{f}\right)=F_2^*\left(\frac{f^k}{f}\right)$ and $W=(W^1,\ldots,W^k)$ we see that 
\begin{equation}\label{F1_is_F2}
 F_1=Z^{-1}\circ W=F_2
\end{equation}
holds on the open neighborhood $F_1^{-1}(\Omega)\cap F_2^{-1}(\Omega)$ of $p$. It follows that $p\in B$. Consequently $B$ is nonempty, open and closed set. Thus $B=M$ as $M$ is connected. 

Let us then assume that $F_1$ and $F_2$ agree to second order on $\Gamma\subset \p M$. Let $p\in \Gamma$ and let $Z$ be conformal harmonic coordinates on a neighborhood of the boundary point $q:=F_1(p)=F_2(p)$. As in the argument above, we have
\[
 L_g(c_j^{(n-2)/4}F_j^*f^k)=L_g(c_j^{(n-2)/4}F_j^*f)=0,
\]
on an open subset of $M$, $j=1,2$ and $k=1,\ldots,n$. Since $F_1$ and $F_2$ agree to second order on the boundary, we have that 
\[
 c_1|_{\Gamma}=c_2|_{\Gamma} \text{ and } \nabla c_1|_{\Gamma}=\nabla c_2|_{\Gamma},
\]
since $c_j=g^{-1}F_j^*h$, $j=1,2$. Therefore, the Cauchy data of $c_1^{(n-2)/4}F_1^*f^k$ and $c_1^{(n-2)/4}F_1^*f$ on $\Gamma$ agree with those of $c_2^{(n-2)/4}F_2^*f^k$ and $c_2^{(n-2)/4}F_2^*f$ respectively. 
Thus, by elliptic unique continuation we have $c_1^{(n-2)/4}F_1^*f^k=c_2^{(n-2)/4}F_2^*f^k$ and $c_1^{(n-2)/4}F_1^*f=c_2^{(n-2)/4}F_2^*f^k$. It follows from the argument above first that $F_1=F_2$ on an open subset, and consequently that $F_1=F_2$ on $M$. \qedhere

\end{proof}

\subsection{Conformal curvature tensors in conformal harmonic coordinates}\label{Bach_formulas}
We show that common conformal curvature tensors in conformal geometry become elliptic operators for the determinant normalized metric in conformal harmonic coordinates. Let us recall definitions of the conformal curvature tensors. 
We write $R_{abcd}$, $R_{ab}$, and $R$ for the Riemann curvature tensor, Ricci tensor, and scalar curvature, respectively. The Schouten tensor is defined as 
\[
P_{ab} = \frac{1}{n-2} \left( R_{ab} - \frac{R}{2(n-1)} g_{ab} \right).
\]
The Weyl tensor of $(M,g)$ is the $4$-tensor 
\[
W_{abcd} = R_{abcd} + P_{ac} g_{bd} - P_{bc} g_{ad} + P_{bd} g_{ac} - P_{ad} g_{bc},
\]
the Cotton tensor is the $3$-tensor 
\begin{equation}\label{Cotton_tensor}
C_{abc} = \nabla_a P_{bc} - \nabla_b P_{ac},
\end{equation}
and, if $n\geq 4$, the Bach tensor is the $2$-tensor
\begin{equation}\label{Bach_tensor}
B_{ab}=\nabla^c\nabla_aP_{bc}-\nabla^c\nabla_cP_{ab}+P^{cd}W_{acdb}.
\end{equation}
The Bach tensor can also be written as
\begin{equation}\label{Bach_via_Weyl}
 B_{ab}=\nabla^c\nabla^dW_{acbd} +\frac{1}{2}R^{cd}W_{acbd}.
\end{equation}
We also consider the Fefferman-Graham obstruction tensors $\mathcal{O}=\mathcal{O}_{(n)}$ having the property
\begin{equation}\label{Obstruction_tensor}
\mathcal{O}_{ab}=\frac{1}{n-3}\Delta^{n/2-2}B_{ab}+\text{lower order terms,}
\end{equation}
where $\Delta=\nabla^a\nabla_a$ and $n \geq 4$ is even. 
If $n=4$, then $\mathcal{O}_{ab} = B_{ab}$. These tensors have the following behavior under conformal scaling in various dimensions: $W(cg) = c W(g)$ for $n \geq 4$, $C(cg) = C(g)$ for $n=3$. The Bach~\eqref{Bach_tensor} and obstruction~\eqref{Obstruction_tensor} tensors are defined in any dimension $n\geq 4$. If the dimension $n$ of the manifold is even and $\geq 4$, then the obstruction tensor $\mathcal{O}_{(n)}$ satisfies
\[
 \mathcal{O}(cg) = c^{-\frac{n-2}{2}} \mathcal{O}(g).
\]
We refer to~\cite{AcheViaclovsky,Bach,Besse,Der,FG,FGbook,Tian} for additional information on the conformal curvature tensors above. 

In~\cite{LS2} it was explained how the above conformal curvature tensors can be regarded as distributions. 
Especially, if in a coordinate system the components of the metric tensor $g_{jk}$ satisfy $g_{jk}\in C^1$, $g_{jk}\in C^2$ or $g_{jk}\in C^{n-1}$, then $W_{abcd}$, $C_{abc}$ and $\mathcal{O}_{ab}$ can be regarded as distributions respectively. We refer to~\cite{LS2} for the details. 

We begin with a proposition which says that in conformal harmonic coordinates the scalar curvature $R(\hat{g})$ of the determinant normalized metric
\[
\hat{g}=\frac{g}{\abs{g}^{1/n}}
\]
does not contain second order derivatives of the metric $\hat{g}$. That is, the highest order part of the scalar curvature of $\hat{g}$ vanishes in conformal harmonic coordinates.
We find this a little surprising and we comment the result after the proof of the proposition. If $g$ is a metric tensor, we use the notation
\begin{equation}\label{Tk_definition}
 T_k(g)
\end{equation}
to denote an unspecified polynomial of the components of $g$, components of the inverse of $g$ and derivatives of $g$ up to order $k\in \mathbb{N}$. That is $T_k(g)=P(g_{ab}, g^{ab}, \nabla g_{ab}, \ldots, \nabla^k g_{ab})$, where $P$ is a polynomial. 

\begin{Proposition}\label{prop_scalar_zero}
 In conformal harmonic coordinates the scalar curvature of the determinant normalized metric $\hat{g}$ does not contain second order derivatives of $\hat{g}$:
 \begin{equation}\label{eq_scalar_zero}
R(\hat{g})=T_1(\hat{g}).
 \end{equation}
\end{Proposition}

\begin{proof}
Let us first recall some identities that hold in general coordinates $(x^a)$ for a general Riemannian metric $g$. We have the general formula for the Laplacian of a logarithm:
\begin{equation}\label{delta_of_log}
\Delta_g\log f=\frac 1f \Delta_g f + \abs{d\log f}_g^2.
\end{equation}
Here $g$ is any Riemannian metric and $f$ any positive $C^2$ function. We have the coordinate formula~\eqref{coordinate_formula_for_Ricci} for the Ricci tensor
\begin{align*}
R_{ab}(g) &=\partial_c \Gamma_{ab}^c - \partial_a \Gamma_{cb}^c + \Gamma_{ab}^c \Gamma_{dc}^d - \Gamma_{cb}^d \Gamma_{ad}^c \\
&=\partial_c \Gamma_{ab}^c - \Gamma_{cb}^d \Gamma_{ad}^c - \frac{1}{2} \partial_{a}\p_b(\log\abs{g}) + \frac{1}{2} \Gamma_{ab}^c \partial_c(\log\abs{g}),
\end{align*}
which holds in any system of local coordinates. Here $\Gamma_{ab}^c=\Gamma_{ab}^c(g)$. Consequently, for the scalar curvature of the determinant normalized metric $\hat{g}$ we have
 \begin{align}\label{scalar_curvature_of_normalized}
 R(\hat{g})&=\hat{g}^{ab}R_{ab}(\hat{g})
 =\hat{g}^{ab}\big(\partial_c \Gamma_{ab}^c(\hat{g})-\frac{1}{2}\partial_a\p_b\log{\abs{\hat{g}}}\big)+T_1(\hat{g}) \nonumber\\
 &=\p_c(\hat{g}^{ab}\Gamma_{ab}^c(\hat{g}))-(\p_c\hat{g}^{ab})\Gamma_{ab}^c(\hat{g})+T_1(\hat{g})=\p_c\Gamma^c(\hat{g}) + T_1(\hat{g})\nonumber \\
 &=\p_c(\hat{g}^{cd}\Gamma_d(\hat{g}))+T_1(\hat{g})=\hat{g}^{cd}\p_c\Gamma_d(\hat{g})+T_1(\hat{g}).
 \end{align}
 Here and in the rest of the proof indices are lowered, raised and contracted by using the metric $\hat{g}$.
 
 Let us then calculate in conformal harmonic coordinates. Proposition~\ref{Z_cond_conf_inv} applied for $c=\abs{g}^{-1/n}$ shows that $\hat{g}$ satisfies the coordinate condition
 \begin{equation}\label{gamma_for_g_hat}
  \Gamma_a(\hat{g})=2\p_a\log f_{\hat{g}},
 \end{equation}
where $f_{\hat{g}}$ satisfies
\[
 L_{\hat{g}}f_{\hat{g}}=\Delta_{\hat{g}}f_{\hat{g}}+\frac{n-2}{4(n-1)} R(\hat{g})f_{\hat{g}}=0. 
\]
The Laplace-Beltrami operator satisfies
\begin{equation}\label{delta_for_normalized}
\hat{g}^{ab}\p_a\p_b=-\Delta_{\hat{g}}+\Gamma^a(\hat{g})\p_a.
\end{equation}
                                                                                                                                                                           By using~\eqref{delta_of_log},~\eqref{gamma_for_g_hat} and~\eqref{delta_for_normalized}, we have that
\begin{align}\label{mystique_equation}
 \hat{g}^{ab}\p_a\Gamma_b(\hat{g})&=2\hat{g}^{ab}\p_a\p_b\log f_{\hat{g}}=-2\Delta_{\hat{g}}\log f_{\hat{g}}+ 2\Gamma^a(\hat{g})\p_a\log f_{\hat{g}} \nonumber\\
 &= -2 \frac{1}{f_{\hat{g}}}\Delta_{\hat{g}}f_{\hat{g}} -2 \abs{d\log f_{\hat{g}}}^2_{\hat{g}}+ 2\Gamma^a(\hat{g})\p_a\log f_{\hat{g}}\nonumber\\
 &= -2 \frac{1}{f_{\hat{g}}}\Delta_{\hat{g}}f_{\hat{g}}-\frac 12\Gamma_a(\hat{g})\hat{g}^{ab} \Gamma_a(\hat{g}) + \Gamma^a(\hat{g})\Gamma_a(\hat{g})\nonumber\\
 &= 2\frac{n-2}{4(n-1)}R(\hat{g})+ \frac 12\Gamma^a(\hat{g})\Gamma_a(\hat{g}).
 \end{align}
 Combining~\eqref{mystique_equation} with~\eqref{scalar_curvature_of_normalized} shows that
 \[
 \frac{n}{2(n-1)}R(\hat{g}) + T_1(\hat{g})=0.
 \]
 The claim follows.
\end{proof}
\color{black}
We comment the result of Proposition~\ref{prop_scalar_zero}. In defining conformal harmonic coordinates we find a positive solution $f$ to the equation $L_gf=0$, which can be seen as solving a local Yamabe problem for zero scalar curvature. Especially, by the equation~\eqref{scalar_curvature_scaled2}, the conformal metric 
\[
 \tilde{g}=f^{p-2}g, \quad p=\frac{2n}{n-2},
\]
has zero scalar curvature, $R(f^{p-2}g)=0$, in any coordinate system. 

However, in conformal harmonic coordinates, which are constructed with respect to the metric $g$, the determinant normalized metric $\hat{g}$ is not generically the same as the metric $f^{p-2}g$. An example is given by letting $(M,g)$ to be a Riemannian manifold whose scalar curvature vanishes. Then any harmonic coordinates are conformal harmonic coordinates by Proposition~\ref{equivalence}, where the function $f\equiv 1$ in the definition of conformal harmonic coordinates. Let $U$ be harmonic coordinates such that $\det{g}$ is not identically one in these coordinates. (For the existence, see~\cite{DK}, or apply Lemma~\ref{prescribed_sols} by noticing that now $L_g=\Delta_g$.)
%
Then we have
\[
 \hat{g}\neq g=f^{p-2}g
\]
in the conformal harmonic coordinates $U$, where $R(f^{p-2}g)=0$.   
Nevertheless we still have $R(\hat{g})=T_1(\hat{g})$. We do not find the interpretation of Proposition~\ref{prop_scalar_zero} obvious because of this reasoning.



We also point out that the equation~\eqref{eq_scalar_zero} is not a tensorial equation for $g$. This is because $\hat{g}$ is not tensorial in $g$ in the sense that
\[
 T^*\hat{g}=J^{2/n}\,\widehat{T^*g}
\]
for a general coordinate transformation $T$. (The reason for the transformation rule is that the determinant of $g$ is a half density, see e.g.~\cite{Du96}.) Here $J$ is the Jacobian determinant of $T$ and $\widehat{T^*g}=T^*g/(\abs{T^*g}^{1/n})$. 


\subsubsection{Results for Bach and obstruction tensors}\label{sec:Bach_and_obs_form}
We consider next the ellipticity properties of the Bach and the obstruction tensors in conformal harmonic coordinates. These tensors have invariances under both coordinate transformations and conformal scalings. Let us first discuss what do we mean when we say that a conformal curvature tensor can be regarded as an elliptic operator.

Consider the Bach tensor $B(\hat{g})$ for the determinant normalized metric $\hat{g}$ in local coordinates. The expression for $B(\hat{g})$ is achieved by substituting $\hat{g}$ into the formula~\eqref{Bach_tensor} of the Bach tensor.
The expression of $B(\hat{g})$ can be thought as a quasilinear partial differential operator applied to $\hat{g}$:
\begin{equation}\label{quasilinear_form}
B(\hat{g})=\sum_{|\alpha|=4}A^\alpha(\hat{g},\hat{g}^{-1},\nabla \hat{g}, \ldots, \nabla^{3}\hat{g})\p_\alpha \hat{g} + T_{3}(\hat{g}).
\end{equation}
The coefficients $A^\alpha$ are polynomials of their arguments and $\alpha=(\alpha_1,\ldots,\alpha_n)$ is a multi-index. The expression $T_{3}(\hat{g})$ was defined in~\eqref{Tk_definition}.

We apply the coordinate condition~\eqref{scaled_Cs} of conformal harmonic coordinates for the conformal metric $\hat{g}$ to have the formula
\begin{equation}\label{cc}
\Gamma_a(\hat{g})=2\p_a\log f_{\hat{g}}
\end{equation}
We substitute~\eqref{cc} to the part of the expression~\eqref{quasilinear_form} of $B(\hat{g})$ that contains fourth order derivatives of $\hat{g}$. 
The function $f=f_{\hat{g}}$ in~\eqref{cc}
depends \emph{non-locally and implicitly} on $\hat{g}$ via the equation
\begin{equation}\label{non-local_f}
L_{\hat{g}}f_{\hat{g}}=0,
\end{equation}
but in general we do not know much about the exact form of the dependence. 

In the coordinate expression of the Bach tensor, the contracted Christoffel symbols $\Gamma_a$ in fact appear only under the operation by the Laplace operator $\Delta$ (or in terms of lower order derivatives of the metric). Now, in conformal harmonic coordinates, a contracted Christoffel symbol satisfies $\Gamma_a=2\p_a\log f$, and the Laplacian of $\log f$ satisfies $\Delta \log f=-\frac{n-2}{4(n-1)}R\ +$ a term containing lower order derivatives of the metric. For the latter see ~\eqref{delta_of_log} and~\eqref{non-local_f}. Note that now $\Delta \log f$ is a partial differential operator acting on $g$. By these remarks, $B(\hat{g})$ can be considered as a quasilinear operator acting on $\hat{g}$ even after substituting the non-local coordinate condition of conformal harmonic coordinates to the part of $B(\hat{g})$ containing derivatives of order $4$.

For the Weyl curvature tensor the situation is different. Imposing the coordinate condition of conformal harmonic coordinates for the Weyl tensor leads to a situation where one needs to deal directly with the complicated dependence of $f$ on $g$. This is because the contracted Christoffel symbols do not appear under the operation by the Laplacian in the formula of the Weyl tensor. However, the Weyl tensor can still be considered as a quasilinear overdetermined elliptic operator in conformal harmonic coordinates when acting on the determinant normalized metric. We will get back to this matter before Theorem~\ref{regularity_thm_interior}.

Let us move on to the calculation of the principal symbols of the Bach and obstruction tensors. These tensors were defined at the beginning of Section~\ref{Bach_formulas}, where it was also explained how these tensors are regarded as distributions for non-smooth Riemannian metrics. The principal symbol of a quasilinear operator is defined as the principal symbol (see e.g.~\cite{T1}) of the linearization of the operator. For example, the notation $\sigma(B_{ab}(\hat{g}))$ means the principal symbol of the operator achieved by  linearizing $B_{ab}$ at $\hat{g}$.



\begin{Proposition}\label{symbol_for_Bach}
Let $(M,g)$ be an $n$-dimensional Riemannian manifold, $n \geq 4$. In the conformal harmonic coordinates hold: 
 \begin{enumerate}
  \item[(a)] Assume that $g\in C^r$, $r>3$, then the coordinate representation of the Bach tensor for the determinant normalized metric can be regarded as a quasilinear operator 
  \[
   B(\hat{g})_{ab}=-\frac{1}{2}\Delta_{\hat{g}}^2\,\hat{g}_{ab}+T_3(\hat{g})
  \]
whose principal symbol $\sigma(B_{ab}(\hat{g}))$ satisfies
 \[
 \sigma(B_{ab}(\hat{g}))h=-\frac{1}{2}\abs{\xi}^4h_{ab}.
 \]
 \item[(b)] If $g\in C^r$, $r>n-1$, then the coordinate representation of the Fefferman-Graham obstruction tensor for the determinant normalized can be regarded as a quasilinear operator
 \[
  \mathcal{O}(\hat{g})_{ab}=-\frac{1}{2(n-3)}\Delta_{\hat{g}}^{n/2}\,\hat{g}_{ab}+T_{n-1}(\hat{g})
 \]
whose principal symbol $\sigma(\mathcal{O}_{ab}(\hat{g}))$ satisfies
 \[
 \sigma(\mathcal{O}_{ab}(\hat{g}))h=-\frac{1}{2(n-3)}\abs{\xi}^nh_{ab}.
 \]
 \end{enumerate}
 Here $\hat{g}$ is the determinant normalized metric and $\abs{\,\cdot\,}=\abs{\,\cdot\,}_{\hat{g}}$.
\end{Proposition}
\begin{proof}
If $g\in C^r$, $r>3$, then by Theorem~\ref{Z-coord}, we have that $g_{jk}\in C^r$ in conformal harmonic coordinates. Consequently $\hat{g}\in C^r$ and the Bach tensor of $\hat{g}$ can be regarded as a distribution by the discussion at the beginning of this section.

The Bach tensor is defined for a general Riemannian metric tensor $g$ and in general coordinates by the formula~\eqref{Bach_tensor}
\[
B_{ab}=\nabla^c\nabla_aP_{bc}-\nabla^c\nabla_cP_{ab}+P^{cd}W_{acbd},
\]
where
\[
P_{ab} = \frac{1}{n-2} \left( R_{ab} - \frac{R}{2(n-1)} g_{ab} \right).
\]
By using the Bianchi identity $\nabla^aR_{ab}=\frac{1}{2}\nabla_bR$ we have that
\begin{align}\label{bach_formula}
 (n-2)B_{ab}&=\frac{1}{2}\nabla_a\nabla_bR-\frac{1}{2(n-1)}\nabla_a\nabla_bR-\nabla^c\nabla_cR_{ab}+\frac{1}{2(n-1)}(\nabla^c\nabla_cR)g_{ab} \nonumber \\
 \quad \quad & + T_3 +(n-2)P^{cd}W_{acbd}.
\end{align}
Here term $T_3=T_3(g)$ results from commuting covariant derivatives. Note that $P^{cd}W_{acbd}=T_2(g)$. The notation $T_k(g)$ was defined in~\eqref{Tk_definition}.

Let us then calculate in conformal harmonic coordinates and apply~\eqref{bach_formula} to the determinant normalized metric. By Proposition~\ref{prop_scalar_zero}, we have the condition 
\begin{equation}\label{scalar_curvature_first_ord}
 R(\hat{g})=T_1(\hat{g}).
\end{equation}
Thus we have that the Bach tensor for $\hat{g}$ in conformal harmonic coordinates satisfies  
\[
B(\hat{g})_{ab}=\widehat{\nabla}^c\widehat{\nabla}_cR(\hat{g})_{ab} +T_3(\hat{g}).
\]
Here $\widehat{\nabla}$ means the covariant derivative calculated with respect to the determinant normalized metric $\hat{g}$, and indices are raised, lowered and contracted with respect to $\hat{g}$. The coordinate formula for the Ricci curvature of $\hat{g}$ can be written by~\eqref{Ricci_deturck_form} as
\begin{equation}\label{Ricci_formula}
 R(\hat{g})_{ab}=-\frac{1}{2}\Delta_{\hat{g}} \hat{g}_{ab}+\frac{1}{2}(\p_a\Gamma(\hat{g})_b+\p_b\Gamma(\hat{g})_a) +T_1(\hat{g}).
\end{equation}
In conformal harmonic coordinates we have by Proposition~\ref{Z_cond_conf_inv} applied in the case $c=\abs{g}^{-1/n}$ that 
\begin{equation}\label{conformal_Gamma}
 \Gamma_a(\hat{g})=2\p_a\log f_{\hat{g}}.
\end{equation}
where $f_{\hat{g}}$ 
satisfies 
\[
 \Delta_{\hat{g}}f_{\hat{g}}=-\frac{n-2}{4(n-1)}R({\hat{g}})f_{\hat{g}}.
\]
Consequently, by using the formula~\eqref{delta_of_log}, we have that
\begin{align}\label{delta_log2}
 2\Delta_{\hat{g}}\log f_{\hat{g}}&=2\frac{1}{f_{\hat{g}}}\Delta_{\hat{g}} f_{\hat{g}}+2\abs{d\log f_{\hat{g}}}_{\hat{g}}^2=-\frac{n-2}{2(n-1)}R({\hat{g}})+\frac{1}{2}\Gamma^a({\hat{g}})\Gamma_a({\hat{g}}) \nonumber \\
 &=-\frac{n-2}{2(n-1)}R(\hat{g})+T_1(\hat{g}).
\end{align}
Combining the equations~\eqref{conformal_Gamma} and~\eqref{delta_log2}, and by using~\eqref{delta_for_normalized}
%
shows that
%
%

\begin{align}\label{Laplacian_of_Gamma}
 \Delta_{\hat{g}} \Gamma_a(\hat{g})&= -\hat{g}^{bc}\p_b\p_c\Gamma_a(\hat{g}) +\Gamma^a(\hat{g})\p_a \Gamma_a(\hat{g}) \nonumber
 =-2\hat{g}^{bc}\p_a\p_b\p_c\log f_{\hat{g}} +T_2(\hat{g})\nonumber \\
 &=-2\p_a(\hat{g}^{bc}\p_b\p_c\log f_{\hat{g}})+2(\p_a\hat{g}^{bc})\p_b\p_c\log f_{\hat{g}}+T_2(\hat{g}) \nonumber\\
 &=2\p_a(\Delta_{\hat{g}}\log f_{\hat{g}})- 2\p_a(\Gamma^b(\hat{g})\p_b\log f_{\hat{g}})+(\p_a\hat{g}^{bc})\p_b\Gamma_c(\hat{g})+T_2(\hat{g}) \nonumber\\
 &=2\p_a(\Delta_{\hat{g}}\log f_{\hat{g}})- \p_a(\Gamma^b(\hat{g})\Gamma_{b}(\hat{g}))+T_2(\hat{g}) \nonumber\\
 &=\p_a\Big(-\frac{n-2}{2(n-1)}R(\hat{g})+T_1(\hat{g})\Big)+T_2(\hat{g}) \nonumber \\
 &=\frac{2-n}{2(n-1)}\p_a R(\hat{g})+ T_2(\hat{g}).
\end{align}

We denote by $\widehat{\nabla}$ the covariant derivative with respect to $\hat{g}$. By applying the formulas ~\eqref{scalar_curvature_first_ord}, ~\eqref{Ricci_formula} and \eqref{Laplacian_of_Gamma} and we have that
\begin{multline*}
\widehat{\nabla}^c\widehat{\nabla}_cR(\hat{g})_{ab}=-\frac{1}{2}\widehat{\nabla}^c\widehat{\nabla}_c(\widehat{\nabla}^d\widehat{\nabla}_d\, \hat{g}_{ab}-(\p_a\Gamma(\hat{g})_b+\p_b\Gamma(\hat{g})_a)) + T_3(\hat{g})\\
=-\frac{1}{2}\Delta^2_{\hat{g}}\,\hat{g}_{ab}+\frac{2-n}{2(n-1)}\p_a\p_bR(\hat{g}) + T_3(\hat{g})=-\frac{1}{2}\Delta_{\hat{g}}^2\,\hat{g}_{ab}+ T_3(\hat{g}).
\end{multline*}
Thus we have that
\begin{equation}\label{Bach_is_Delta_square}
B_{ab}(\hat{g})=-\frac{1}{2}\Delta_{\hat{g}}^2\,\hat{g}_{ab} +T_3(\hat{g})
\end{equation}
and consequently the principal symbol satisfies
\[
\sigma(B_{ab}(\hat{g}))h=-\frac{1}{2}\abs{\xi}_{\hat{g}}^4h_{ab}.
\]

The claim regarding the Fefferman-Graham obstruction tensor $\mathcal{O}$ follows from the definition of the obstruction tensor
\[
 \mathcal{O}_{ab}(\hat{g})=\frac{1}{n-3}\Delta_{\hat{g}}^{n/2-2}B_{ab}(\hat{g})+T_{n-1}(\hat{g})
\]
and by using the equation~\ref{Bach_is_Delta_square}.
\end{proof}
We record a formula from the proof for later reference. We have in conformal harmonic coordinates
\begin{equation}\label{delta_of_gamma}
 \Delta_{\hat{g}}\Gamma_a(\hat{g})=T_2(\hat{g}),
\end{equation}
which follows by combining~\eqref{scalar_curvature_first_ord} and~\eqref{Laplacian_of_Gamma}. We note that this condition is similar to the condition $\Gamma_a=0$ of harmonic coordinates.
%

We next calculate a formula for the Cotton tensor for a determinant normalized metric in conformal harmonic coordinates. We show that in this situation the Cotton tensor can be regarded as overdetermined elliptic operator. An overdetermined elliptic partial differential operator is a partial differential operator whose principal symbol is injective. 

\begin{Proposition}\label{prop:Cotton}
Let $(M,g)$ be an $n$-dimensional Riemannian manifold, $n \geq 3$. In conformal harmonic coordinates hold:
 \begin{equation*}
  C_{abc}(\hat{g})=-\frac{1}{2(n-2)}\Delta_{\hat{g}}(\widehat{\nabla}_a\hat{g}_{bc}-\widehat{\nabla}_b\hat{g}_{ac}) +T_2(\hat{g}),
 \end{equation*}
 where $\widehat{\nabla}$ is the covariant derivative with respect to $\hat{g}$.
 
 The Cotton tensor can be regarded as an overdetermined elliptic operator for the determinant normalized metric in conformal harmonic coordinates.
\end{Proposition}
\begin{proof}
 By Proposition~\ref{prop_scalar_zero} and by the equation~\eqref{Ricci_formula} we have that
 \begin{align}\label{R_and_Ricci}
  R(\hat{g})&=T_1(\hat{g}) \nonumber \\
  R_{ab}(\hat{g})&=-\frac{1}{2}\Delta_{\hat{g}}\,\hat{g}_{ab}+\frac{1}{2}(\p_a\Gamma_b(\hat{g})+\p_b\Gamma_a(\hat{g}))+T_1(\hat{g}) .
 \end{align}
 We ease the notation for the rest of the proof by regarding all quantities depending on a Riemannian metric to be calculated with respect to $\hat{g}$ without further notice. By using the definition of the Cotton tensor, see~\eqref{Cotton_tensor}, and the equations ~\eqref{R_and_Ricci} we have that 
 \begin{align}\label{cotton_manipulate}
 (n-2)C_{abc}&=\nabla_a(R_{bc}-\frac{1}{2(n-1)}R\,\hat{g}_{bc})-\nabla_b(R_{ac}-\frac{1}{2(n-1)}R\,\hat{g}_{ac}) \nonumber\\
 &=\nabla_a R_{bc}-\nabla_b R_{ac} +T_2(\hat{g}) \\ 
 &=-\frac{1}{2}\nabla_a(\Delta\,\hat{g}_{bc}-(\p_b\Gamma_c+\p_c\Gamma_b))\nonumber \\
 &\quad\quad\quad+\frac{1}{2}\nabla_b(\Delta\,\hat{g}_{ac}-(\p_a\Gamma_c+\p_c\Gamma_a)) +T_2(\hat{g})\nonumber \\
 &=-\frac{1}{2}\Delta(\nabla_a\,\hat{g}_{bc}-\nabla_b\,\hat{g}_{ac})+ \frac{1}{2}\p_c(\p_a\Gamma_b-\p_b\Gamma_a)+T_2(\hat{g}). 
 \end{align}
 We next use the coordinate condition of conformal harmonic coordinates~\eqref{conformal_f} 
 \[
  \Gamma_a=2\p_a\log(f_{\hat{g}}) 
 \]
  for the conformal metric $\hat{g}$ to conclude be commuting partial derivatives that
 \[
   \p_a\Gamma_b-\p_b\Gamma_a=0.
 \]
Thus the equation~\ref{cotton_manipulate} reads
\begin{equation}\label{cotton_manipulate2}
 (n-2)C_{abc}=-\frac{1}{2}\Delta(\nabla_a\hat{g}_{bc}-\nabla_b\hat{g}_{ac})+T_2(\hat{g}).
\end{equation}

 To show that the Cotton tensor is an elliptic operator for the determinant normalized metric in conformal harmonic coordinates, we continue to manipulate the formula~\eqref{cotton_manipulate2} for the Cotton tensor. We write the leading order part of~\eqref{cotton_manipulate2} by adding a vanishing term $\log{\abs{\hat{g}}}=0$  as
 \begin{equation}\label{adhoc_gauge}
 -\frac{1}{2}\Delta\left(\nabla_a\hat{g}_{bc}-\frac{1}{n}\nabla_a(\log{\abs{\hat{g}}})\hat{g}_{bc}-\nabla_b\hat{g}_{ac}\right).
 \end{equation}
 We show that the operator 
 \begin{equation}\label{in_brackets}
  \hat{g}_{ab}\mapsto \nabla_a\hat{g}_{bc}-\frac{1}{n}\nabla_a(\log{\abs{\hat{g}}})\hat{g}_{bc}-\nabla_b\hat{g}_{ac}
 \end{equation}
 is a quasilinear overdetermined elliptic operator acting on $\hat{g}$.
 For this, we need to show that the principal symbol
 \begin{equation}\label{princip_symb}
  h_{ab}\mapsto \xi_ah_{bc}-\frac{1}{n}\xi_a h\hat{g}_{bc} -\xi_bh_{ac}
 \end{equation}
 of the linearization of the operator~\eqref{in_brackets} is injective. Here $h_{ab}$ is a symmetric matrix and $\xi\in \R^n\setminus \{0\}$. We have also denoted $h=\hat{g}^{ab}h_{ab}$ and used 
 \[
  \nabla_a(\log \abs{\hat{g}})=\hat{g}^{bc}\nabla_a\hat{g}_{bc}.
 \]
To show the injectivity, we set the symbol in~\eqref{princip_symb} to zero: 
 \begin{equation}\label{cotton_symbol}
 \xi_ah_{bc}-\frac{1}{n}\xi_a h\hat{g}_{bc} =\xi_bh_{ac}.
 \end{equation}
 Contracting the equation~\eqref{cotton_symbol} with $\hat{g}^{bc}$ and $\xi^a$ and $\hat{g}^{ac}\xi^b$ sequentially yield the equations
 \begin{align*}
 0&=h(\xi)_a, \quad a=1,\ldots,n \\
 \abs{\xi}^2(h_{bc}-\frac{1}{n}h\hat{g}_{bc})&=\xi_bh(\xi)_c, \quad b,c=1,\ldots,n \\
 h(\xi,\xi)-\frac{1}{n}\abs{\xi}^2h&=\abs{\xi}^2h.
 \end{align*}
 Here we have denoted $h(\xi)_a=h_{ab}\xi^b$, $h(\xi,\xi)=h_{ab}\xi^a\xi^b$ and $\xi^a=\hat{g}^{ab}\xi_b$.
 Together the first and the last of these equation shows that $h=0$. By substituting $h=0$ to the middle equation and by using the first equation again yields
 \[
 h_{ab}=0, \quad a,b=1,\ldots,n.
 \]
 Thus the principal symbol is injective, and we can regard the Cotton tensor as an overdetermined elliptic for the determinant normalized metric in conformal harmonic coordinates.
\end{proof}
In the proof of Proposition~\ref{prop:Cotton} above, we added the term $\frac{1}{n}\nabla_a(\log{\abs{\hat{g}}})\hat{g}_{bc}$ in~\eqref{adhoc_gauge} in a seemingly arbitrary way. One way to interpret this step is by considering an example where the Cotton tensor satisfies an equation $C(\hat{g})=0$ in conformal harmonic coordinates. In this case $\hat{g}$ is a symmetric positive definite matrix field, which satisfies the overdetermined system
 \begin{align*}
  0&=C_{abc}(\hat{g})=-\frac{1}{2}\Delta_{\hat{g}}(\widehat{\nabla}_a\hat{g}_{bc}-\widehat{\nabla}_b\hat{g}_{ac})+T_2(\hat{g}) \\
  0&=\Delta_{\hat{g}}\widehat{\nabla}_a\log(\hat{g})=\hat{g}^{bc}\Delta_{\hat{g}}\widehat{\nabla}_a\hat{g}_{bc}+T_2(\hat{g}).
 \end{align*}
 This system for $\hat{g}$ can be seen to be overdetermined elliptic by calculations similar to those done in the proof of Proposition~\ref{prop:Cotton}.

 
 
We will next prove regularity results for conformal curvature tensors. We have shown that Bach, Fefferman-Graham obstruction and Cotton tensors are elliptic in conformal harmonic coordinates for the determinant normalized metric. As discussed before in this section, the situation for the Weyl tensor is more complicated. 
To prove a regularity result for the Weyl tensor, we recall the formula 
\[
 B_{ab}=\nabla^c\nabla^dW_{acbc} +\frac{1}{2}R^{cd}W_{acbd}.
\]
This formula together with Proposition~\ref{symbol_for_Bach} imply that the principal symbol of the Weyl tensor is injective for a determinant normalized metric in conformal harmonic coordinates: If $\sigma(W_{acbd}(\hat{g}))h=0$, then $\xi^c\xi^d\sigma(W_{acbd}(\hat{g}))h=\sigma(\nabla^c\nabla^d W_{acbd}(\hat{g}))h=\sigma(B_{ab}(\hat{g}))h=0$ and thus $h_{ab}=0$. Therefore, we are able to consider the Weyl tensor as an overdetermined elliptic operator. 
\begin{Theorem}\label{regularity_thm_interior}
Let $(M,g)$ be an $n$-dimensional Riemannian manifold without boundary, and let $g \in C^r$, in some system of local coordinates. 
\begin{enumerate}
\item[(a)]
If $n \geq 4$, $r>2$, and if $W_{abc\phantom{d}}^{\phantom{abc}d}$ is in $C^{s}$, $C^\infty$ or $C^\omega$, for some $s> r-2$, $s\notin \Z$, in conformal harmonic coordinates, then $\abs{g}^{-1/n} g_{ab}$ is in $C^{s+2}$, $C^\infty$ or $C^\omega$ respectively in these coordinates. 

\item[(b)]
If $n = 3$, $r>2$, and if $C_{abc}$ is in $C^{s}$, $C^\infty$ or $C^\omega$, for some $s > r-3$, $s\notin \Z$, in conformal harmonic coordinates, then $\abs{g}^{-1/n} g_{ab}$ is in $C^{s+3}$, $C^\infty$ or $C^\omega$ respectively in these coordinates.

\item[(c)]
If $n \geq 4$ is even, $r > n-1$, and if $\abs{g}^{\frac{n-2}{2n}}\mathcal{O}_{ab}$ is in $C^s$, $C^\infty$ or $C^\omega$, for some $s > r-n$, $s\notin \Z$, in conformal harmonic coordinates, then $\abs{g}^{-1/n} g_{ab}$ is in $C^{s+n}$, $C^\infty$ or $C^\omega$ respectively in these coordinates. Here $\mathcal{O}=\mathcal{O}_{(n)}$.

\end{enumerate}

Especially, in the situations above, if a tensor $W$, $C$ or $\mathcal{O}$ vanishes, then $\abs{g}^{-1/n}g$ is real analytic $C^\omega$ in conformal harmonic coordinates.

\end{Theorem}
\begin{proof}
We argue in conformal harmonic coordinates. We have that $g$ is in $C^r$ in these coordinates. 
By using the conformal invariance of conformal curvature tensors, we have that $W_{abc\phantom{d}}^{\phantom{abc}d}(\hat{g})$, $C_{abc}(\hat{g})$ and $\mathcal{O}_{ab}(\hat{g})$ are in $C^s$. 
These tensors can be considered as quasilinear elliptic operators for the determinant normalized metric by Propositions~\ref{symbol_for_Bach} and~\ref{prop:Cotton}. The ellipticity of Weyl tensor was discussed above. The claims of the theorem follow then by applying elliptic regularity results for linear elliptic systems together with a bootstrap argument.  Details for the proof for the cases where $s\in \R$ can be found from the proof of~\cite[Theorem 1.2]{LS2}. If the determinant normalized metric $\hat{g}_{ab}=\abs{g}^{-1/n} g_{ab}$ has enough regularity initially, the claims also follow from~\cite[Theorem 41]{Besse}. By first proving that $\hat{g}_{ab}\in C^k$, where $k\in \mathbb{N}$ is large enough, the case $s=\omega$ follow from~\cite[Theorem 41]{Besse}. See also~\cite[Theorem 6.7.6 and 6.8.1]{Mo66} and~\cite{ADN64}.
\end{proof}

Boundary regularity results for conformal curvature tensors involve defining proper boundary conditions for higher order nonlinear elliptic systems. Discussing these boundary conditions is outside the scope of this work. We expect that natural boundary regularity results for conformal curvature tensors can be proven by using conformal harmonic coordinates. Our expectation is based on the facts that the conformal harmonic coordinates are constructed by solving well-behaved linear elliptic equations and that the conformal invariance is fixed by explicit normalization of the determinant of the metric. We refer to 
~\cite{Gr08, Mo66} for discussions regarding boundary conditions for higher order elliptic systems. We also mention that boundary regularity for the Ricci tensor in harmonic coordinates has been studied before, see e.g.~\cite{AKKLT04}.

 Next we give a couple of applications of the real analytic regularity results in Theorem~\ref{regularity_thm_interior}. 
 The first result says that if on a Bach or obstruction flat manifold a point has a neighborhood conformal to the Euclidean space, then the same property holds for all points of the manifold. Recall that if the Weyl tensor vanishes on a neighborhood of a point in dimension $\geq 4$, there exists isothermal coordinates on neighborhood of the point, where $g_{ab}=c(x)\delta_{ab}$.
 \begin{Theorem}\label{ucp_for_Bach_flat}
  Let $(M,g)$ be an obstruction flat, $\mathcal{O}_{(n)}\equiv 0$, connected Riemannian manifold without boundary of even dimension $n\geq 4$ with $g\in C^r$, $r>n-1$. Assume that $W(g)=0$ on an open set of $M$. Then $(M,g)$ is locally conformally flat i.e.
  \[
   W(g)\equiv 0 \text{ on } M.
  \]
 \end{Theorem}
 \begin{proof}
  Let $B\subset M$ be the largest open set where $W=0$. By assumption $B\neq \emptyset$. Let $p\in \p B$ and let $Z$ be conformal harmonic coordinates on a neighborhood $\Omega$ of $p$. Since the obstruction tensor $\mathcal{O}_{(n)}(g)$ vanishes in $M$, we have by Theorem~\ref{regularity_thm_interior} that
  \begin{equation}\label{hatg_real_analytic}
   \hat{g}_{ab}\in C^\omega(\Omega).
  \end{equation}
  Since the Weyl tensor is conformally invariant, we have that 
  \[
   W(\hat{g})=0 \text{ on } B\cap \Omega. 
  \]
  Since the Weyl tensor is a polynomial of the components of the metric, its inverse and its derivatives up to second order, we also have by Equation~\ref{hatg_real_analytic} that
  \[
   W(\hat{g})\in C^\omega(\Omega).
  \]
  Thus $W(\hat{g})=0$ in $\Omega$ by real analyticity. Since $W$ is conformally invariant, it follows that $W(g)=0$ in $\Omega$ and thus $ \Omega\subset B$. Especially we have that $p\in B$, which shows that $B$ is closed. We have proven that $B$ is open, closed and nonempty, and thus $B=M$ since $M$ is connected. 
\end{proof}

The Cartan-Janet theorem~\cite{Ja26, Ca27} states that a real analytic Riemannian manifold can be locally isometrically embedded into $\R^{n(n+1)/2}$. Since an obstruction flat metric is real analytic up to a conformal factor, it is more natural to consider local conformal embeddings into a Euclidean space. In this case the dimension of the Euclidean space can also be lowered by one:
\begin{Theorem}\label{conformal_Cartan_Janet}
 An obstruction flat Riemannian manifold $(M,g)$, $\mathcal{O}_{(n)}\equiv 0$, of even dimension $\geq 4$ with $g\in C^{r}$, $r>n-1$, can be locally embedded in $\R^N$, $N=n(n+1)/2-1$, by a conformal mapping.
\end{Theorem}
\begin{proof}
Let $p\in M$ and let $Z$ be conformal harmonic coordinates on a neighborhood $\Omega$ of $p$. Since $(M,g)$ is obstruction flat, we have that $\hat{g}\in C^\omega$ by Theorem~\ref{regularity_thm_interior}. By the Cartan-Janet theorem for conformal embeddings~\cite[Theorem 2]{JM73} there exists a mapping $I:(B,\hat{g})\to \R^N$, where $B$ is an open subset $Z(\Omega)\subset \R^n$, such that $I^*e=c\,\hat{g}$ where $c$ is a positive function and $N=n(n+1)/2-1$. Here $e$ is the Euclidean metric. Then $I\circ Z$ is a local conformal embedding of $(M,g)$ to $(\R^N,e)$ with $\overline{c}:=(c\,\abs{g}^{-1/n})|_Z$ as the conformal factor:
\[
 (I\circ Z)^*e=Z^*I^*e=Z^*(c\,\hat{g})=(c\,\abs{g}^{-1/n})|_Z\, g=\overline{c}\,g. \qedhere
\]
\end{proof}

\section{Conformal wave coordinates in Lorentzian geometry}\label{sec_lorentzian_manifolds}
In this final section we briefly discuss conformal wave coordinates on Lorentzian manifolds. We define these coordinates in analogy to the conformal harmonic coordinates in the Riemannian setting. We assume for simplicity that the manifolds and the Lorentzian metrics considered in this section are $C^\infty$ smooth. We also assume that the considered manifolds have no boundary and the dimensions of the manifolds are at least $3$. We keep the exposition short and refer to~\cite{Ri09} for standard definitions in Lorentzian geometry.

The conformal Laplacian has a direct analogue on a Lorentzian manifold called the \emph{conformal wave operator}
\[
 \mathcal{L}_g = \square_g + \frac{n-2}{4(n-1)}R(g),
\]
where $\square_gu=\nabla^a\nabla_au=-|g|^{-1/2}\p_a\left(|g|^{1/2}g^{ab}\p_bu\right)$ and $\abs{g}=-\det(g)$.
The conformal wave operator has the same conformal invariance properties as its Riemannian counterpart, see e.g.~\cite{CG18}. 
\begin{Definition}(Conformal wave coordinates)\label{def:Z_coords_Lorentzian}
 Let $(M,g)$ be a Lorentzian manifold of dimension $n\geq 3$. Local coordinates $(Z^1,\ldots,Z^n)$ on an open set $\Omega$ are called conformal wave coordinates if the components functions $Z^k$, $k=1,\ldots,n$, are functions on $\Omega$ of the form
\begin{equation}
Z^k=\frac{f^k}{f},
\end{equation}
where $f^k$ and $f>0$ satisfy $\mathcal{L}_gf^k=\mL_gf=0$.
\end{Definition}
We prove the existence of conformal wave coordinates.
\begin{Proposition}\label{Z-coord_lorentz}
Let $(M,g)$ be Lorentzian manifold, $\dim(M)=n\geq 3$. Let $p\in M$. There exists a $C^\infty$ smooth local coordinate system $Z=(Z^1,\ldots,Z^n)$ on a neighborhood of $p$ whose coordinate functions $Z^k$, $k=1,\ldots,n$, are of the form
\begin{align}\label{Z_form}
Z^k&=\frac{f^k}{f}, \\
\mL_gf^k&=\mL_gf=0, \quad f>0. \nonumber
\end{align}
\end{Proposition}
\begin{proof}
In the Lorentzian setting finding solutions $f^k$ and $f$ such that $Z=(f^1/f,\ldots, f^n/f)$ is a coordinate chart is more straightforward than in the Riemannian setting. This is because the conformal wave operator is hyperbolic and we may thus prescribe Cauchy data of solutions of the conformal wave operator on a spacelike  Cauchy hypersurface. 

Let $p\in M$. We first recall that all Lorentzian manifolds can be regarded locally as globally hyperbolic manifolds~\cite[Theorem 2.7]{Mi19}. Thus there is a neighborhood $\Omega$ of $p$ with coordinates $(t,x')$ and a local $C^\infty$ smooth spacelike Cauchy surface $S=\{t=0\}\subset \Omega$ such that we may find a $C^\infty$ smooth solution $u$ to the equation
\begin{align}\label{conf_wave_equation}
 \mL_g u(t,x')&=0, \quad (t,x')\in \Omega \\
 u(0,x')&=u_0(x') \text{ and } \p_tu(0,x')=u_1(x'), \quad x'\in S
\end{align}
where $u_0$ and $u_1$ are any $C^\infty$ smooth functions defined on $S=\{t=0\}$, see e.g.~\cite[Theorem 3.2.11]{BGP07}. Let us construct the functions $f^k$ and $f$ as follows. We let $f^1$ solve~\eqref{conf_wave_equation} with $f_1(0,x')\equiv 0$ and $\p_tf^1(0,x')\equiv 1$. Denote $x'=(x'_2,\ldots,x'_n)$. We let $f^k$, $k=2,\ldots,n$, solve~\eqref{conf_wave_equation} with $f^k(0,x')=x'_k$ and $\p_tf^k(0,x')\equiv 0$ and we let $f$ solve~\eqref{conf_wave_equation} with $f^k(0,x')\equiv 1$ and $\p_tf^1(0,x')\equiv 0$. Then we have $DZ|_S=I_{n\times n}$ by the calculation in~\eqref{differential_of_Z}. Thus $Z$ is invertible on a neighborhood of $(t,x')=(0,0)$ by the inverse function theorem.
\end{proof}

The behavior of conformal wave coordinates under conformal scalings and mappings is the same as for conformal harmonic coordinates. Especially statements of Propositions~\ref{equivalence} and~\ref{gauge_cond_of_Z} converted for conformal wave coordinates remain true. We also have
\[
 \Gamma(g)_a=2\p_a\log f,
\]
 where  $\mathcal{L}_gf=0$ in conformal wave coordinates. 
 In analogy to Corollary~\ref{Z_coords_conformally_invariant}, conformal wave coordinates for a Lorentzian metric $g$ are conformal wave coordinates for any conformal metric $cg$. 
We have the same formulas for curvature tensors in conformal wave coordinates for the determinant normalized metric $\hat{g}$:
 \begin{align*}
  R(\hat{g})&=T_1(\hat{g}) \\
  \mathcal{O}(\hat{g})_{ab}&=-\frac{1}{2(n-3)}\Box_{\hat{g}}^{n/2}\hat{g}+T_{n-1}(\hat{g}) \text{ if }n\geq 4 \\
  C(\hat{g})_{abc}&=-\frac{1}{2}\Box_{\hat{g}}(\widehat{\nabla}_a\hat{g}_{bc}-\widehat{\nabla}_b\hat{g}_{ac}) +T_2(\hat{g}) \text{ if }n\geq 3,
 \end{align*}
 where $a,b,c=1,\ldots,n$. The equation  
 \[
 \Delta_{\hat{g}}\Gamma_a(\hat{g})=T_2(\hat{g})
 \]
 is also satisfied in conformal wave coordinates. See Section~\ref{Bach_formulas} for the formulas.

A major difference of the Lorentzian setting to the Riemannian case is that the smoothness of conformal wave coordinates can vary. Not all conformal wave coordinates are $C^\infty$ smooth even if $g$ is $C^\infty$ smooth. This has consequences. To give an example, the argument we used to prove a regularity result for conformal mappings in the Riemannian setting fails in the Lorentzian setting; 
Using a conformal mapping to pull back smooth conformal wave coordinates does not produce conformal wave coordinates that are automatically smooth. We mention that regularity of conformal mappings in the Lorentzian setting was proven by Hawking~\cite{HKM76}.

We finish with a unique continuation result of conformal mappings in Lorentzian spacetimes.  
We defined in Section~\ref{reg_ucp_for_conf} before Theorem~\ref{ucp_for_conformal} what it means that two mappings on a manifold agree to second order on a subset of the manifold. 
We refer to~\cite{Ri09} on the standard definitions appearing in the result and in its proof.

\begin{Theorem}\label{ucp_for_conformal_Lorentzian}
 Let $(M,g)$ be a connected, oriented, time oriented, globally hyperbolic Lorentz manifold without boundary of dimension $n\geq 3$, and let $(N,h)$ be an $n$-dimensional Lorentz manifold without boundary. Let $S$ be a $C^\infty$ smooth spacelike Cauchy hypersurface in $M$. 
 Let $\Omega$ be a subset of $S$.
 
 Assume that $F_1$ and $F_2$ are two $C^\infty$ smooth locally diffeomorphic conformal mappings $M \to N$, 
 $F_1^*h=c_1 g$ and $F_2^*h=c_2 g$. Assume also that $F_1$ and $F_2$ agree to second order on $\Omega$. 
Then, for $p\in \mathcal{D}^+(\Omega)\cup \mathcal{D}^-(\Omega)$ holds 
 \[
  F_1(p)=F_2(p). 
 \]
 Here $\mathcal{D}^+(\Omega)$ and $\mathcal{D}^-(\Omega)$ are the future and past Cauchy developments of $\Omega$.
%
Especially if $\Omega=S$, then $F_1=F_2$ in $M$. 
\end{Theorem}
\begin{proof}
The proof is similar to that of Theorem~\ref{ucp_for_conformal}. However, the simple topological argument used there to show that the set where the claim holds is open and closed, and thus the whole manifold, is replaced here by an argument that uses that $M$ is globally hyperbolic. For the latter, we argue similarly as in the proof of~\cite[Corollary 12.12]{Ri09}. 
We show that $F_1=F_2$ on $\mathcal{D}^+(\Omega)$. The proof for $\mathcal{D}^-(\Omega)$ follows from an analogous argument.

It is sufficient to proof that $F_1=F_2$ holds in the interior of $\mathcal{D}^+(\Omega)$. This is because $\mathcal{D}(\Omega)^+\setminus \Omega\subset \overline{\text{Int}(\mathcal{D}(\Omega)^+)}$ by~\cite[Proof of Corollary 12.12]{Ri09}. Thus, if $F_1=F_2$ on $\text{Int}(\mathcal{D}^+(\Omega))$, then by the continuity of $F_1$ and $F_2$, it follows that $F_1=F_2$ on $\mathcal{D}(\Omega)^+\setminus \Omega$. By assumption $F_1|_\Omega=F_2|_\Omega$. 

Let $P\in \text{Int}(\mathcal{D}^+(\Omega))$. Then, by~\cite[Lemma 40, p.\,423]{On83}, the set 
\[
 K:=J^-(P)\cap \mathcal{D}^+(\Omega)
\]
is compact. Let $t$ be a temporal function such that all its level sets are spacelike Cauchy
hypersurfaces and $t^{-1}(0)=S$, see~\cite[Theorem 11.27]{Ri09}. 
Given an interval $I\subset \R$ and $t_0\in \R$, we set
\[
R_I=t^{-1}(I)\cap K, \quad S_{t_0}=t^{-1}(t_0) 
\]
and 
\[
      R_{t_0}=S_{t_0}\cap K.
\]
Define $T=t(P)>0$. We will prove that the set $I_B$ of $s\in [0,T)$, such that $F_1$ and $F_2$ agree to second order on $R_{[0,s]}$, is the set $[0,T)$. 
Let $t_0\in [0,T)$ and assume that $F_1=F_2$ agree to second order on $R_{t_0}$. 
Let $q\in R_{t_0}$ and let $Z$ be conformal wave coordinates on a neighborhood $\Omega_N$ of $F_1(q)=F_2(q)$. Let $\Omega_M=F_1^{-1}(\Omega_N)\cap F_2^{-1}(\Omega_N)$. 
We choose an open neighborhood $U_q$ of $q$ contained in $\Omega_M$,
\[
 q\in U_q\subset\Omega_M,
\]
such that if $z\in U_q$, which lies to the future of $S_{t_0}$, there are geodesic normal coordinates on $V$ centered at $z$ such that $J^{-}(z)\cap J^+(S_{t_0})$ is compact and contained in $V$. Such an open set $U_q$ can be found by modifying the argument used to prove of~\cite[Lemma 12.7]{Ri09} to apply for the Lorentz manifold $\Omega_M$. (The proof of~\cite[Lemma 12.7]{Ri09} only uses that a causal curve can intersect a Cauchy hypersurface only once. Causal curves in $\Omega_M$ can intersect $S_{t_0}$ only once.) 
The reason $U_q$ is chosen this way, is that we will apply~\cite[Lemma 12.10]{Ri09} later. 
 
 Since $R_{t_0}$ is compact, we find a finite open covering $\{U_{q_r}\}_{r=1}^L$ of $R_{t_0}$ with the following properties: For each $r\in \{1,\ldots,L\}$, we have that there are $q_r\in R_{t_0}$, an open neighborhood $U_{q_r}$ of $q_r$ and conformal wave coordinates on a neighborhood $\Omega_{N,r}\subset N$ of the point $F_1(q_r)=F_2(q_r)$ such that $q_r\in U_{q_r}\subset F_1^{-1}(\Omega_{N,r})\cap F_2^{-1}(\Omega_{N,r})$. 
Let us denote $U=\cup_{r=1}^L U_{q_r}$. We have that there is $\eps>0$ such that $R_{[t_0,t_0+\eps]} \subset U$, see~\cite[Proof of Corollary 12.12]{Ri09}, if $R_{t_0}\subset U$ and $R_{t_0}$ non-empty.


Let $z\in R_s$ for some $s\in [t_0,t_0+\eps]$. By the construction of $U$, we have that $z\in U_{q_r}$ for some $r\in \{1,\ldots,L\}$. 
Let $Z$ be the conformal wave coordinates on $\Omega_{N,r}$ as described above. 
Thus $Z=(Z^1,\ldots,Z^n)=(f^1/f,\ldots,f^n/f)$ and we have on $U_{q_r}$ that 
%
\[
  F_j^*\left(\frac{f^k}{f}\right)=\frac{F_j^*f^k}{F_j^*f}=\frac{c_j^{(n-2)/4}F_j^*f^k}{c_j^{(n-2)/4}F_j^*f}.
\]
As in the proof of Theorem~\ref{ucp_for_conformal}, we have on $U_{q_r}$ that
\[
\mathcal{L}_g\big(c_j^{(n-2)/4}F_j^*f^k\big)=\mathcal{L}_g\big(c_j^{(n-2)/4}F_j^*f)=0,
 \]
for $j=1,2$ and $k=1,\ldots,n$. 
 %
 Since $F_1$ and $F_2$ agree to second order on $R_{t_0}$, we have that
\[
 c_1|_{R_{t_0}}=c_2|_{R_{t_0}} \text{ and } \nabla c_1|_{R_{t_0}}=\nabla c_2|_{R_{t_0}},
\]
since $c_j=g^{-1}F_j^*h$ for $j=1,2$. It follows that the values and gradients of the functions $c_1^{(n-2)/4}F_1^*f^k$ and $c_2^{(n-2)/4}F_2^*f^k$ agree on $R_{t_0}$. Note that $J^{-}(z)\cap S_{t_0}\subset R_{t_0}$: Let $w\in J^-(z)\cap S_{t_0}$. Since $w\in J^{-}(z)$ and $z\in \mathcal{D}^+(\Omega)$, we have that $w\in \mathcal{D}^+(\Omega)$.  
Since $w\in J^{-}(z)$ and $z\in J^{-}(P)$, we have that $w\in J^{-}(P)$. Thus $w\in S_{t_0}\cap J^-(P)\cap \mathcal{D}^+(\Omega)=R_{t_0}$.
 
 By the above and by~\cite[Lemma 12.10]{Ri09} we have that 
 \[
  c_1^{(n-2)/4}F_1^*f^k=c_2^{(n-2)/4}F_2^*f^k \text{ on } J^{-}(z)\cap J^+(S_{t_0}).
 \]
We have similarly $c_1^{(n-2)/4}F_1^*f=c_2^{(n-2)/4}F_2^*f$ on $J^{-}(z)\cap J^+(S_{t_0})$. 
By denoting $W_j=F_j^*Z$, $j=1,2$, we have that
\[
 F_1=Z^{-1}\circ W_1=Z^{-1}\circ W_2=F_2 
\]
on $J^{-}(z)\cap J^+(S_{t_0})$. 
Thus $F_1$ and $F_2$ agree to second order at $z$. Since $s\in [t_0,t_0+\eps]$ and $z\in R_s$ were arbitrary, we have that $F_1$ and $F_2$ agree to second order on $R_{[t_0,t_0+\eps]}$. 

The set $I_B$ of $s\in [0,T)$ such that $F_1$ and $F_2$ agree to second order on $R_{[0,s]}$ is non-empty by assumption, closed by $C^2$ smoothness of $F_1$ and $F_2$ and open by the argument above. 
Thus $I_B=[0,T)$ and we have that 
\[
 F_1(P)=F_2(P).
\]
Thus $F_1=F_2$ on $\mathcal{D}^+(\Omega)$.  If $\Omega=S$, then $\mathcal{D}(\Omega)=M$. \qedhere

\end{proof}

\appendix
\section{List of formulas in Riemannian geometry}\label{list_of_formulas}

We list formulas in Riemannian geometry.
The Christoffel symbols are given by 
\[
\Gamma_{ab}^c = \frac{1}{2} g^{cd}(\partial_a g_{bd} + \partial_b g_{ad} - \partial_d g_{ab}).
\]
Contracted Christoffel symbols are defined as 
\[
\Gamma^a = g^{bc} \Gamma_{bc}^a, \qquad \Gamma_a = g_{ab} \Gamma^b.
\]
Noting the identity $g^{bc} \partial_a g_{bc} = \partial_a(\log\abs{g})$, we see that 
\begin{equation}\label{stnd_contracted}
\Gamma^a = -\partial_b g^{ab} - \frac{1}{2} g^{ab} \partial_b(\log\abs{g}), \qquad 
\Gamma_a = g^{bc} \partial_b g_{ac} - \frac{1}{2} \partial_a(\log\abs{g}). 
\end{equation}
This also implies that 
\begin{equation}\label{gamma_contracted_log}
\Gamma_{ba}^b = \frac{1}{2} \partial_a(\log\abs{g}).
\end{equation}
The Ricci tensor, $R_{bc} = R_{abc}^{\phantom{abc}a}$, where $R_{abc}^{\phantom{abc}d}$ is the Riemann curvature tensor, is given by 
\begin{align}\label{coordinate_formula_for_Ricci}
R_{ab} &= \partial_c \Gamma_{ab}^c - \partial_a \Gamma_{cb}^c + \Gamma_{ab}^c \Gamma_{dc}^d - \Gamma_{cb}^d \Gamma_{ad}^c \nonumber\\
 &= \partial_c \Gamma_{ab}^c - \frac{1}{2} \partial_{ab}(\log\abs{g}) - \Gamma_{cb}^d \Gamma_{ad}^c  + \frac{1}{2} \Gamma_{ab}^c \partial_c(\log\abs{g}). 
\end{align}
The latter equality follows from~\eqref{gamma_contracted_log}. The Ricci curvature can also be written as
\begin{equation}\label{Ricci_deturck_form}
 R_{ab}=-\frac{1}{2}\Delta g_{ab}+\frac{1}{2}\left(\p_a\Gamma_b+\p_b\Gamma_a\right) + \text{ lower order terms.}
\end{equation}
See e.g.~\cite[Lemma 4.1]{DK}.


The conformal Laplacian transforms under a conformal scaling of the metric $g$ by a positive function $c$ as
\begin{equation}\label{conformal_invariance_of_L2}
 L_{cg} u = c^{-\frac{n+2}{4}} L_g (c^{\frac{n-2}{4}} u).
\end{equation}
By defining
\[
p=\frac{2n}{n-2},
\]
then the conformal invariance can also be written as
\begin{equation}\label{conformal_invariance_of_L}
L_{c^{p-2}g}(c^{-1}u)=c^{1-p}L_gu.
\end{equation}
The scalar curvature transforms under a conformal scaling by $c$ as
\begin{equation}\label{scalar_curvature_scaled2}
 R(c^{p-2}g)=c^{1-p}\left(4\frac{n-1}{n-2}\Delta_gc+R(g)c\right)=4\frac{n-1}{n-2}c^{1-p}L_gc=0.
 \end{equation}


\section{Construction of local solutions for the conformal Laplace equation}\label{local_solutions}
We construct local solutions to $L_gu=0$ with some prescribed properties at a given point. The first part of the lemma below can also be found from~\cite[p. 228, Theorem 1]{BJS}. The proofs of both parts of the lemma are based on the fact that the maximum principle holds for $L_g$ in small enough domains and on a relatively standard scaling argument. 

\begin{Lemma}\label{prescribed_sols}
Let $\Omega\subset\R^n$ be a neighborhood of the origin.

{\bf (1)} Let $g\in C^r(\ol{\Omega})$, $r>2$, be a Riemannian metric on $\ol{\Omega}\subset \R^n$. Let $\sigma\in \R^n$. Then there exists $\eps>0$ 
and a function $f^\sigma\in C^{r}(\overline{B}(0,\eps))$ that solves
\[
 L_{g}f^\sigma=0 \mbox{ on } \overline{B}(0,\eps),
\]
with
\[
 df^\sigma(0)=\sigma \mbox{ and } f^\sigma(0)=0.
\]
Also, there is a function $f\in C^r(\overline{B}(0,\eps))$ such that $L_gf=0$ and $f>0$ on $\overline{B}(0,\eps)$ with $f(0)=1$.

{\bf (2)}  Let $g \in C^r$, $r>2$, be a Riemannian metric on the set $\ol{\Omega}\cap\{x\in \R^n: x^n\geq 0\}$. Let $D=\overline{B}(0,1)\cap \{x\in \R^n: x^n\geq 0\}\subset \R^n$ and $\Gamma=D\cap \{x\in \R^n: x^n=0\}$, and set $D_\eps=\{\eps x: x\in D\}$ and $\Gamma_\eps=D_\eps\cap \{x\in \R^n: x^n=0 \}$. There exists $\eps>0$ and a function $f^n\in C^{r}(\overline{D}_\eps)$ that solves
\begin{align*}
 L_{g}f^n&=0, \mbox{ in } D_\eps, \\
 f^n &= 0, \mbox{ on } \Gamma_\eps,
\end{align*}
with
\[
 df^n(0)=dx^n.
\]
Also, there is a function $f\in C^r(\overline{D}_\eps)$ such that $L_gf=0$ and $f>0$ on $\overline{D}_\eps$ and $f=1$ on $\Gamma_\eps$. 
\end{Lemma}
\begin{proof}
{\bf(1)} 
 The conformal Laplace equation for a function $u\in C^2$ reads
 \begin{equation}\label{cLaplace_in_coords}
 -\frac{1}{|g|^{1/2}}\p_a\left(|g|^{1/2}g^{ab}\p_bu\right)+R(g)u=0,
 \end{equation}
 where $R(g)$ is a polynomial of the components of $g$ and $g^{-1}$ and derivatives of the components of $g$ up to order $2$. Since $r>2$, the equation~\ref{cLaplace_in_coords} is an elliptic equation for $u$ whose coefficients are in the H\"older class $C^{r-2}$ with $r-2>0$. 

We first find a function $f$ that satisfies 
 \[
  L_gf=0, \text{ with } f(0)=1
 \]
 on a neighborhood $\Omega\subset \R^n$ of the origin. Such a function $f$ can be found by first considering a small enough neighborhood of the origin so that the maximum principle holds for $L_g$, see e.g.~\cite[Proposition 1.1]{BNV94}. After that one may apply the Fredholm alternative, see e.g.~\cite[Theorem 6.15]{GT}, to solve a Dirichlet problem with a positive boundary value. By scaling one may fix $f(0)=1$. (Alternatively, see~\cite[p.228, Thm. 1]{BJS} for the existence.) The solution $f$ is $C^{r}$ regular. Let $\delta>0$ so that $f>0$ on $\overline{B}(0,\delta)\subset \Omega$.  We scale the metric as
 \[
  \tilde{g}=f^{p-2}g\in C^r.
 \]
 Then we have that
 \[
  L_{\tilde{g}}=\Delta_{\tilde{g}}
 \]
 since $R(\tilde{g})=0$ by the equation 
  \begin{equation*}
 R(\tilde{g})=f^{1-p}\left(4\frac{n-1}{n-2}\Delta_gf+R(g)f\right)=4\frac{n-1}{(n-2)}f^{1-p}L_gf=0.
 \end{equation*}

 Let $k\in \{1,\ldots,n\}$ and let $x^k$ be the corresponding coordinate function. Since the operator $L_{\tilde{g}}$, which is defined for functions on $\overline{B}(0,\delta)$, has no zeroth order term, the Dirichlet problem for $L_{\tilde{g}}$ on any $\overline{B}(0,\eps)$ is uniquely solvable for any $C^1$ boundary value function for any $\eps\in (0,\delta]$. We refer to~\cite[Theorem 8.9]{GT} for this standard result. 
 For $\eps\in (0,\delta]$, we solve the Dirichlet problem
\begin{align}
 L_{\tg}u_\eps&=0, \mbox{ on } B(0,\eps), \\
 u_\eps &=x^k \mbox{ on } \p B(0,\eps) \nonumber.
\end{align}

Define a family of functions $\tilde{u}_\eps$, $\eps\in (0,\delta]$ on the domain $\overline{B}(0,1)$ by
\[
 \tilde{u}_\eps(x)=\frac{1}{\eps}u_\eps(\eps x).
\]
We have that $\tilde{u}_\eps$ solves the equation
\begin{align}\label{scaled_solution}
 - \tg^{jk}&(\eps x)\p_{jk}\tilde{u}_\eps(x)+\eps \tilde{\Gamma}^l(\eps x)\p_l\tilde{u}_\eps(x) \\
 &=-\frac{\eps^2}{\eps} (\tg^{jk}\p_j\p_k u_\eps)(\eps x)+\frac{\eps^2}{\eps} (\tilde{\Gamma}^l\p_l u_\eps)(\eps x)=\eps (\Delta_{\tg}u_\eps)(\eps x)=0 \nonumber
\end{align}
on the fixed domain $\overline{B}(0,1)$. Here $\tilde{\Gamma}^l=\tg^{ab}\tilde{\Gamma}_{ab}^l$, where $\tilde{\Gamma}_{ab}^l$ are the Christoffel symbols with respect to the scaled metric $\tg\in C^r$.

Define a family $L^\eps$, $\eps\in (0,\delta]$, of operators acting on functions on $\overline{B}(0,1)$ by
\[
 L^\eps v:=-\tg^{jk}(\eps x)\p_{jk}v(x)
 +\eps \tilde{\Gamma}^l(\eps x)\p_lv.
\]
Let us also write
\[
 \tilde{u}_\eps(x)=x^k+w_\eps(x).
\]
If $x\in \p B(0,1)$, we have that $\tilde{u}_\eps(x)=u_\eps(\eps x)/\eps=(\eps x)^k/\eps=x^k$. Using this and the equation~\ref{scaled_solution} we have that $w_\eps$ solves
\begin{align*}
 L^\eps w_\eps(x)&=-L^\eps x^k=-\eps \tilde{\Gamma}^k(\eps x), \quad x\in B(0,1) \\
 w_\eps(x)&=0, \quad \quad\quad\quad\quad\quad\quad\quad\quad x\in \p B(0,1).
\end{align*}

Let us denote
\[
 r_\eps(x)=-\eps \tilde{\Gamma}^k(\eps x)
\]
Since the operator $L^\eps$ has no zeroth order term we have by ~\cite[Theorem 8.16]{GT} (by taking $\lambda=1$, $\nu^2=\eps \sum_l\abs{\tilde{\Gamma^l}}^2_\infty$, $g=r_\eps$) that
\[
 \sup_{\overline{B}(0,1)}\abs{w_\eps}\leq \sup_{\p B(0,1)}\abs{w_\eps}+C_1\norm{r_\eps}_{C^{0}(\overline{B}(0,1))}=C_1\norm{r_\eps}_{C^{0}(\overline{B}(0,1))}.
\]
Elliptic estimates~\cite[Theorem 8.33]{GT} give for $\alpha\in (0,1)$ that 
\[
 \norm{w_\eps}_{C^{1,\alpha}(\overline{B}(0,1))}\leq C_2(\norm{w_\eps}_{C^{0}(\overline{B}(0,1))}+\norm{r_\eps}_{C^{0}(\overline{B}(0,1))}).
 \]
 The constants $C_1$ and $C_2$ are independent of $\eps$.
Consequently we have that
\begin{equation}\label{correction_small}
 \norm{w_\eps}_{C^{1,\alpha}(\overline{B}(0,1))}\leq C_2(C_1+1) \norm{r_\eps}_{C^{0}(\overline{B}(0,1))}\leq C_3\eps,
\end{equation}
where $C_3$ is independent of $\eps$.

We have
\[
 du_\eps(x)=d(\eps \tilde{u}_{\eps}(x/\eps))=(d\tilde{u}_\eps)(x/\eps)
\]
Combining this with~\eqref{correction_small}, and since $dx^k$ is constant, we have the estimate 
\begin{equation}\label{arbitrary_close}
 \norm{dx^k-du_\eps}_{C^\alpha{(\overline{B}(0,\eps))}}=\norm{dx^k-d\tilde{u}_\eps}_{C^\alpha{(\overline{B}(0,1))}}<C_3\eps.
\end{equation}
By adding a constant to $u_\eps$, we may redefine $u_\eps$ so that $u_\eps(0)=0$ while still having $L_{\tilde{g}}u_\eps=0$ and the estimate~\eqref{arbitrary_close}. 

Let us define 
\[
 f_\eps^k:=fu_\eps.
\]
Then we have on $\overline{B}(0,\eps)$ that 
\[
 0=L_{\tilde{g}}u_\eps=f^{1-p}L_g(fu_\eps)=f^{1-p}L_gf_\eps^k
\]
and
\[
 df_\eps^k(0)=df(0)u_\eps(0)+f(0)du_\eps(0)=du_\eps(0)=dx^k+\mathcal{O}_{C^\alpha}(\eps).
\]
We also have $f_\eps^k(0)=0$. 
Since $k$ and $\eps$ were arbitrary it follows that on a small enough ball $B(0,\eps)$, $\eps>0$, we may approximate any coordinate function $x^k$ by a solution to the equation
\[
 L_gf^k=0
\]
to a small error in $C^1(\overline{B}(0,\eps))$. Let $\sigma\in \R^n$. It follows by the linearity of $L_g$ that we may find a small neighborhood of the origin so that
 \begin{equation}
 L_gf^\sigma=0, \text{ with } f^\sigma(0)=0, \quad df^\sigma(0)=\sigma.
 \end{equation}
We have proven the first part of the lemma.

{\bf (2)} The proof of the second part of the lemma is almost identical to the proof of the first part. We only point out the differences. 

Let $D=\overline{B}(1,0)\cap \{x\in \R^n: x^n\geq 0\}\subset \R^n$. It will be convenient to work with $C^{1,\alpha}$ smooth domains. For this, let $D'\subset D$ be a convex $C^{1,\alpha}$ smooth closed domain, $\alpha\in (0,1)$, such that $D'$ contains the half-ball $B(\gamma,0)\cap \{x\in \R^n: x^n\geq 0\}$, for some $\gamma>0$. (Consider $D'$ to be $D$ with suitably rounded edges.) 
Let us denote 
$D_\eps'=\{\eps x: x\in D'\}\subset \R^n$ and $\Gamma_\eps':=D_\eps'\cap \{x^n=0\}$ for $\eps>0$ small. 

There is $\gamma>\delta>0$ such that there exists a positive function $f$ that solves 
\begin{align*}
 L_gf&=0 \text{ on } D_\delta' \\
 f&=1 \text{ on } \Gamma_\delta'
\end{align*}
The existence of $f$ follows from the fact that the maximum principle holds for $L_g$ on small enough domains~\cite[Proposition 1.1]{BNV94} and by the Fredholm alternative~\cite[Theorem 6.15]{GT}. Define $\tilde{g}=f^{p-2}g$. Since $L_{\tilde{g}}$ has no zeroth order term on $D_\delta$, the family of Dirichlet problems
\begin{align}
 L_{\tg}u_\eps&=0, \mbox{ on } D_\eps', \\
 u_\eps &=x^n \mbox{ on } \p D_\eps' \nonumber.
\end{align}
has a unique solution for $\eps\in (0,\delta]$. The scaling argument of part (1) can be used also here since $D'$ is convex and thus star-shaped with respect to origin.  The scaling argument together with~\cite[Theorem 8.33]{GT} (here we use that $D_\eps'$ is $C^{1,\alpha}$) shows that 
\[
\norm{dx^n-du_\eps}_{C^\alpha{(\overline{D}'_\eps)}}<C\eps.
\]

The functions   
\[
 f_\eps^n:=fu_\eps
\]
parameterized by $\eps$ defined on $\ol{D}_{\eps}'$ satisfy
\[
 L_gf_\eps^n=0 \text{ and } f_\eps^n|_{\Gamma_\eps'}=0.
\]
Since $f^n_\eps$ is zero on the part $\Gamma_\eps'$ of the boundary of $D_\eps'$, we have that $df_\eps^n(0)=\sum_a (\p_af_\eps^n(0))dx^a=(\p_nf_\eps^n(0)) dx^n$ (no summation over $n$ here). It follows that
\begin{align*}
 df_\eps^n(0)&=(\p_nf_\eps^n(0)) dx^n=(f(0)\p_nu_\eps(0)+u_\eps(0)\p_nf(0))dx^n\\
 &=\p_nu_\eps(0)dx^n=(1+\mathcal{O}_{C^\a}(\eps))dx^n.
\end{align*}
Thus for $\eps$ small enough we have by scaling by a constant that there is a function such that $L_gf^n=0$ on $D_\eps'$, with $f^n=0$ on $\Gamma_\eps'$, and which satisfies 
\[
 df^n(0)=dx^n. \qedhere
\]

Since $B(\gamma,0)\cap \{x\in \R^n: x^n\geq 0\}\subset D'$ and because $D'$ is convex, we have that $D_{\gamma\eps}=B(\eps\gamma,0)\cap \{x\in \R^n: x^n\geq 0\}\subset D'_\eps$. Redefining $\eps$ as $\gamma\eps$ proves the claim. 
\end{proof}

\bibliographystyle{alpha}

\end{document}